\documentclass[11pt]{article}

\usepackage{enumitem}

\RequirePackage{newtxtext}
\RequirePackage{newtxmath}
\RequirePackage{bm}
\RequirePackage{endnotes}

\usepackage[font=small,labelfont=bf]{caption}
\usepackage{algorithm}
\usepackage{algpseudocode}
\usepackage{tikz}
\usepackage{tabularx}

\usepackage{subcaption}

\usepackage{natbib}
 \bibpunct[, ]{(}{)}{,}{a}{}{,}%

\def\d{\mathrm{d}}

\newcommand{\E}{\mathbb{E}}
\newcommand{\R}{\mathbb{R}}

\newcommand{\N}{\mathbb{N}}
\newcommand{\p}{\mathbb{P}}

\newcommand{\id}{\mathrm{1}}

\usepackage{graphicx}
\usepackage{amsfonts}
\usepackage{amsmath}
\usepackage{url}
\usepackage{fancyhdr}

\usepackage{indentfirst}

\usepackage{amsthm}
\usepackage{color}
\usepackage{natbib}
\usepackage[colorlinks=true,citecolor=blue]{hyperref}
\usepackage{threeparttable}

\def\d{\mathrm{d}}

\newcommand{\be}{\begin{eqnarray}}
\newcommand{\ee}{\end{eqnarray}}
\newcommand{\by}{\begin{eqnarray*}}
\newcommand{\ey}{\end{eqnarray*}}
\newcommand{\bn}{\begin{enumerate}}
\newcommand{\en}{\end{enumerate}}
\newcommand{\bi}{\begin{itemize}}
\newcommand{\ei}{\end{itemize}}

\renewcommand{\ge}{\geqslant}
\renewcommand{\le}{\leqslant}
\renewcommand{\geq}{\geqslant}
\renewcommand{\leq}{\leqslant}
\renewcommand{\epsilon}{\varepsilon}
\theoremstyle{plain}
\newtheorem{theorem}{Theorem}
\newtheorem{corollary}[theorem]{Corollary}
\newtheorem{lemma}[theorem]{Lemma}

\newtheorem{proposition}[theorem]{Proposition}
\theoremstyle{definition}

\newtheorem{example}{Example}[section]

\theoremstyle{definition}

\numberwithin{equation}{section} \numberwithin{theorem}{section}
\renewcommand{\cite}{\citet}
\renewcommand{\cdots}{\dots}
\newcommand{\ignore}[1]{}

\begin{document}

\title{Reconciling Risk-Aversion Paradoxes in the Distribution-Free Newsvendor Problem: Scarf’s Rule Meets Dual Utility}


\author{Jonathan Yu-Meng Li$^{\dag}$, \;  Tiantian Mao$^{\dag\dag}$, \; Reza Valimoradi$^{\dag}$\\
\\
$^{\dag}$Telfer School of Management\\
University of Ottawa \\
Ottawa, Ontario, Canada K1N 6N5\\
\\
$^{\dag\dag}$ Department of Statistics and Finance, School of Management \\
University of Science and Technology of China\\
 Hefei, Anhui, China}

\maketitle

\begin{center}
\textbf{Abstract}
\end{center}

How should a risk-averse newsvendor order optimally under distributional ambiguity? Attempts to extend Scarf's celebrated distribution-free ordering rule using risk measures have led to conflicting prescriptions: CVaR-based models invariably recommend ordering less as risk aversion increases, while mean–standard deviation models--paradoxically--suggest ordering more, particularly when ordering costs are high. We resolve this behavioral paradox through a coherent generalization of Scarf’s distribution-free framework, modeling risk aversion via distortion functionals from dual utility theory. Despite the generality of this class, we derive closed-form optimal ordering rules for any coherent risk preference. These rules uncover a consistent behavioral principle: a more risk-averse newsvendor may rationally order more when overstocking is inexpensive (i.e., when the cost-to-price ratio is low), but will always order less when ordering is costly. Our framework offers a more nuanced, managerially intuitive, and behaviorally coherent understanding of risk-averse inventory decisions. It exposes the limitations of non-coherent models, delivers interpretable and easy-to-compute ordering rules grounded in coherent preferences, and unifies prior work under a single, tractable approach. We further extend the results to multi-product settings with arbitrary demand dependencies, showing that optimal order quantities remain separable and can be obtained by solving single-product problems independently.


\section{Introduction}\label{sec:Intro}
The newsvendor problem is foundational in inventory theory, balancing overage and underage costs to determine an optimal order quantity. The classical formulation assumes that the demand distribution is known and the decision-maker is risk-neutral. Subsequent research has extended the model in two directions: addressing \textit{distributional ambiguity}, most notably through Scarf’s distribution-free approach (\cite{S58}) based on limited moment information, and incorporating \textit{risk aversion} (e.g., \cite{CR08}), particularly by leveraging the modern theory of risk measures.

The intersection of these two directions—risk aversion under distributional ambiguity—has drawn increasing attention in recent years (e.g., \cite{HKW15}). However, most existing studies have prioritized computational tractability, offering limited insight into how risk preferences reshape ordering behavior—particularly in Scarf-type distribution-free settings. Two notable exceptions are the works of \cite{NST11} and \cite{HDZ14}, which derive closed-form solutions under specific risk measures: CVaR and mean–standard deviation, respectively. Yet their conclusions diverge sharply. While \cite{NST11} find that greater risk aversion invariably leads to smaller orders, \cite{HDZ14} paradoxically show that under high cost-to-price ratios, greater aversion can lead to larger orders (see Figure~\ref{fig:two_panel_with_legend} for a visual illustration of this striking contrast). This apparent contradiction—should more risk aversion lead to ordering more or less?—highlights a deeper gap in our understanding and motivates the need for a unified, behaviorally coherent framework.\\

\begin{figure}[htbp]
  \centering
  \begin{tabular}{@{}c@{\hspace{1em}}c@{}}
    \raisebox{-1cm}{\rotatebox{90}{\textbf{Optimal order}}}
    &
    \begin{minipage}{0.7\textwidth}
      \centering

      \begin{subfigure}[b]{0.44\textwidth}
        \centering
        \caption*{\textbf{CVaR} (\cite{NST11}) }
        \includegraphics[width=\linewidth]{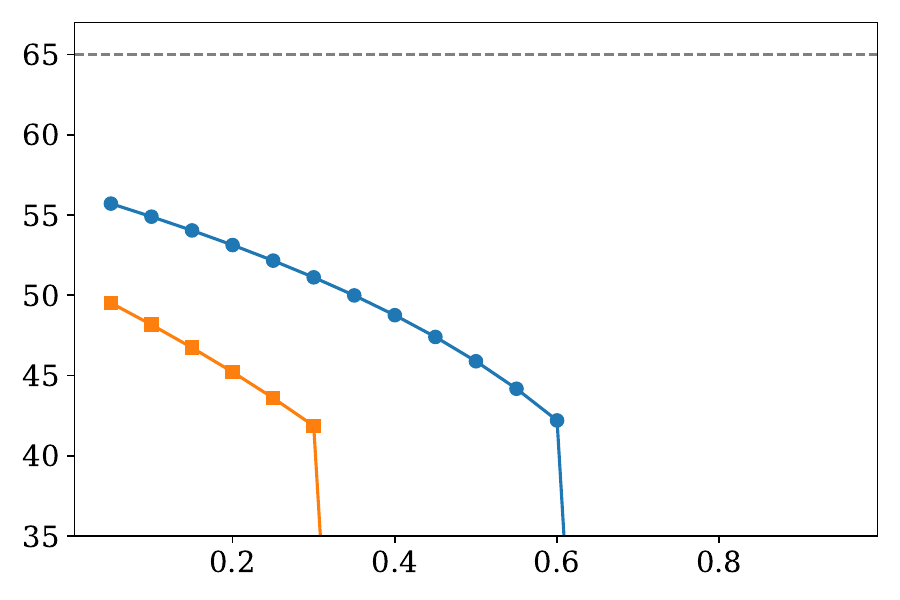}
      \end{subfigure}\hfill
      \begin{subfigure}[b]{0.44\textwidth}
        \centering
        \caption*{\textbf{Mean StD} (\cite{HDZ14})}
        \includegraphics[width=\linewidth]{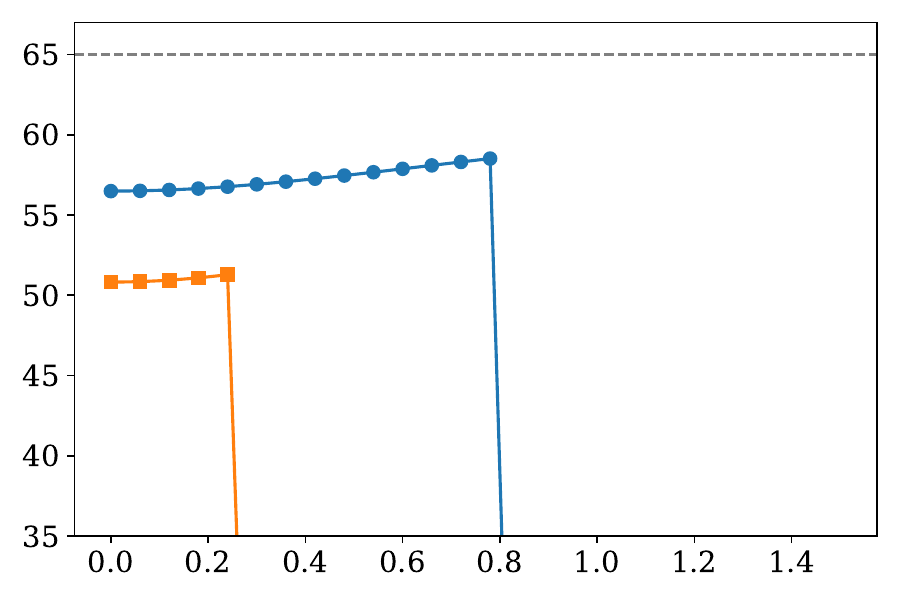}
      \end{subfigure}

      
      \includegraphics[height=0.4cm]{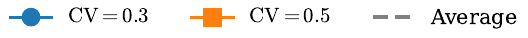}

     
      \textbf{Risk aversion}
    \end{minipage}
  \end{tabular}

  \caption{Optimal order quantities under two risk measures, with a cost-to-price ratio of $0.7$ and a coefficient of variation ${\rm CV} = \sigma / \mu$, where $\mu$ and $\sigma$ denote the mean and standard deviation of demand, respectively.
\textbf{(a) CVaR:} the order quantity decreases with increasing risk aversion.
\textbf{(b) Mean-standard deviation:} the order quantity increases with increasing risk aversion.
}
  \label{fig:two_panel_with_legend}
\end{figure}
\vspace{-1.8em}

We resolve this behavioral paradox by generalizing Scarf’s classical ordering rule through the lens of dual utility theory, specifically via distortion functionals—a rich, axiomatic class of coherent risk measures rooted in Yaari’s dual theory of choice (\cite{Y87}). Under a competing set of axioms to those of expected utility, dual utility theory shows that any coherent risk preference can be represented through a distortion functional, offering exceptional generality. This class includes CVaR as a special case and captures a broad spectrum of behaviorally grounded preferences, all while preserving desirable properties such as monotonicity, subadditivity, and comonotonic additivity. Despite this generality, we derive ordering rules in closed-form for arbitrary coherent preferences. These expressions—though more intricate than Scarf’s original rule—explicitly reveal how the shape of the distortion function governs optimal ordering behavior under ambiguity.

Crucially, our analysis reveals a behavioral pattern that is more nuanced than the monotonic prescriptions of CVaR (\cite{NST11}) and more coherent than the paradoxical recommendations of mean–standard deviation models (\cite{HDZ14})—such as ordering more when costs are high. We find that the relationship between risk aversion and the optimal order quantity hinges fundamentally on the cost-to-price ratio: when overstocking is relatively inexpensive, more risk-averse newsvendors may rationally order more to hedge against lost sales. As the cost-to-price ratio increases, however, ordering behavior transitions toward the conventional pattern—ordering less with greater aversion—thereby reconciling the seemingly conflicting findings in \cite{NST11} and \cite{HDZ14}. While \cite{HDZ14} attribute the anomalous results under the mean–standard deviation model to distributional ambiguity, we show that the deeper issue lies in the incoherence of the risk measure itself. Replacing standard deviation with a coherent alternative—such as deviation from the median—restores behavioral consistency and underscores the importance of coherence in modeling risk preferences under ambiguity. 

Beyond reconciling prior contradictions, our findings also challenge the conventional wisdom in classical risk-averse models without ambiguity (e.g., \cite{GT07} and \cite{CR08}), which predict that greater risk aversion invariably leads to smaller orders. In contrast, our results reveal that this monotonic pattern can break down in the presence of ambiguity: more risk-averse decision-makers may be incentivized to order more when overstocking is sufficiently inexpensive. This richer and more flexible behavioral response highlights the broader insight our framework offers into risk-averse inventory decisions under uncertainty.

Finally, while distribution-free risk-averse newsvendor problems are generally challenging in multi-product settings due to the nonlinearity of risk measures, our dual utility framework extends naturally to such cases with arbitrary demand dependencies. By leveraging the structure of the worst-case distribution, we show that the optimal solution decomposes into independent single-product problems—allowing the behavioral insights and tractability of the single-product setting to carry over seamlessly.

\medskip
\noindent \textbf{Summary of contributions.} This paper offers four main contributions:
\begin{enumerate}
\item We generalize Scarf’s classical rule to a rich, axiomatic class of coherent risk preferences using distortion functionals rooted in dual utility theory.
\item We derive closed-form ordering rules for arbitrary coherent distortion functionals, enabling tractable, interpretable solutions under broad behavioral specifications.
\item Our framework reconciles prior contradictory findings on risk aversion in distribution-free newsvendor models, clarifying when and why ordering behavior may deviate from conventional monotonic patterns.
\item We extend the model to multi-product settings with arbitrary demand dependencies, showing that the optimal solution decomposes across products—preserving both insight and computational simplicity.
\end{enumerate}

\section{Problem Formulation}  
Building on Scarf’s classical distribution-free setup (\cite{S58}), we study a single-period newsvendor problem where the demand distribution is unknown but specified by its first two moments: mean $\mu$ and standard deviation $\sigma$. Let $p$ and $c$ denote the unit resale price and ordering cost, with $p > c > 0$. The decision-maker selects an order quantity $x \geq 0$ before demand $S \sim F$ is realized. The realized profit is
\[ 
\psi_0(x,p,  c) :=  p \min\{S, x\} - c x.
\]
If unsold inventory can be salvaged at a unit value $s > 0$, then the effective resale price and ordering cost adjust to $p - s$ and $c - s$, respectively.\footnote{In this case, the profit becomes $(p - s)\min\{S, x\} - (c - s)x$.} We denote the set of all cumulative distribution functions (cdfs)   on $\mathbb{R}_+$  consistent with the specified moments by
\[
\mathcal{S}(\mu, \sigma) = \left\{ F ~\middle|~ F \text{ is a cdf on } \mathbb{R}_+,~ \mathbb{E}_F[S] = \mu,~ \mathbb{E}_F[S^2] = \mu^2 + \sigma^2 \right\}.
\]

Scarf’s classical solution maximizes the worst-case expected profit over $\mathcal{S}(\mu, \sigma)$, but it implicitly assumes risk neutrality, evaluating uncertainty solely through the mean of the profit. While prior risk-averse extensions—such as those based on CVaR (\cite{NST11}) and mean–standard deviation (\cite{HDZ14})—highlight the need to move beyond risk-neutral expected value, their contradictory prescriptions (as discussed in Section~\ref{sec:Intro}) call for a more coherent and behaviorally grounded perspective.

Seeking such a foundation, we generalize Scarf's distribution-free setting to risk-averse models defined via \emph{distortion functionals}—a class of preference models that reweight probabilities to reflect attitudes toward uncertainty. This generalization is rooted in Yaari’s dual theory of choice (\cite{Y87}), which offers a compelling alternative to expected utility by showing that, under a distinct set of axioms, preferences can be represented through distorted expectations.
\[
\rho_h(X) = \int_0^1 \operatorname{VaR}_\alpha(X)\, \d h(\alpha),
\]
captures the decision-maker's risk preferences. Here, $h: [0,1] \rightarrow [0,1]$ is a non-decreasing, concave function with $h(0) = 0$ and $h(1) = 1$, and $\operatorname{VaR}_\alpha(X) := \inf\{x : \mathbb{P}(X \le x) \ge \alpha\}$. Concavity of $h$ emphasizes lower quantiles, reflecting aversion to unfavorable outcomes. Within this setting, the risk-averse newsvendor seeks to maximize the worst-case distorted profit:
\[
\max_{x \geq 0} \inf_{F \in \mathcal{S}(\mu, \sigma)} \rho^F_h(\psi_0(x, p, c)), \quad \text{(with concave } h).
\]

While the concavity of the distortion function $h$ captures aversion to poor profit outcomes, it is mathematically equivalent---and increasingly common in practice---to frame risk aversion from a loss perspective. This viewpoint connects directly to the theory of coherent risk measures, where distortion-based functionals such as CVaR and other spectral measures serve as canonical examples. In this setting, risk is assessed via a convex distortion function applied to losses, reflecting aversion to severe downsides and ensuring coherence. To formalize this, we define the loss function
\[
\psi(x, p, c) := -\psi_0(x, p, c),
\]
and pose the worst-case risk minimization problem:
\begin{align}
    \label{wc-0}
\rho_h^{\rm wc}(\mu,\sigma):=  \min_{x \geq 0} \sup_{F \in \mathcal{S}(\mu, \sigma)} \rho^F_h(\psi(x, p, c)), \quad \text{(with convex } h).
\end{align}
This formulation serves as the basis for our analysis throughout the remainder of the paper.   Given $x\ge 0$, an optimal distribution in the inner problem of \eqref{wc-0} is referred to as a worst-case distribution.  Table \ref{tab1} presents several widely used distortion risk measures. We denote by  $\mathcal H$ the set of all convex distortion functions and by $h'$ the left-continuous function that equals the left derivative of $h$ on $(0,1]$, with $h'(0)$ defined as the right derivative of $h$ at $0$.  A canonical example is Conditional Value-at-Risk (CVaR), defined for confidence level $\beta \in [0,1]$ as:
$$
\mathrm{CVaR}_{\beta}(X) = \frac{1}{1-\beta} \int_{\beta}^{1} \operatorname{VaR}_{\alpha}(X) \, \mathrm{d}\alpha,~~\beta\in [0,1)~~~{\rm and}~~~\mathrm{CVaR}_{1}(X)=\mathrm{VaR}_{1}(X).
$$
\begin{table}[h!]
\small
\begin{center}
\caption{Examples of Risk Measures and Their Distortion Functions}
\label{tab1}
\begin{tabularx}{\textwidth}{l X X}
\hline
{Risk measures} &{Definition} & {Distortion Function} \\
\hline
Mean-CVaR & 
$
\begin{array}{l}\rho_{\lambda,\alpha}(X) = \lambda \mathbb{E}[X] + (1-\lambda)\mathrm{CVaR}_\alpha(X) \\
\lambda\in [0,1],\alpha\in [0,1]
\end{array}
$
& 
$h(u) = \lambda u + \frac{1-\lambda}{1-\alpha}(u-\alpha)_+$ \\
\hline 
  Dev.  Median   
  &
$  
\begin{array}{l}
   \mathrm {DM}_a(X) = \mathbb{E}[X] + a \mathbb{E}[|X - m_X|]   \\
    a\in [0,1]
\end{array}
$
& $h(u) =\left\{\begin{array}{ll}  (1-a)u, & u<1/2,\\ (1+a)u - a, & u\ge 1/2. \end{array} \right.$    \\
\hline
 Wang transform &
$ 
\begin{array}{l}
W_\lambda(X) = \int_{0}^{1} F_X^{-1}(u)\, \mathrm{d}h_\lambda(u) \\
\lambda\in\R
\end{array}
$
& 
$h_\lambda(u) = 1 - \Phi\left[\Phi^{-1}(1-u) + \lambda\right]$ \\ 
\hline
Prop. Hazards & 
$\mathrm{P}_a(X) = \int_{0}^{1} F_X^{-1}(u)\, \mathrm{d}h(u)$ 
& 
$h(u) = 1 - (1-u)^a$, $0<a\le 1$ \\
\hline
Gini Measure & 
$\mathrm{Gini}_a(X) = \mathbb{E}[X] + a \mathbb{E}[|X - X'|]$ 
& 
$h(u) = (1-a)u + a u^2$, $0<a\le 1$ \\
\hline
\end{tabularx}
\end{center}
\vskip0.2cm
{\small {\it Notes:} Dev.  Median is deviation from the median introduced by \cite{D90}; $m_X = F_X^{-1}(1/2)$; Wang transform is introduced by \cite{W00}; Prop. Hazards is proportional hazards transform introduced by \cite{W95}; $\Phi$ is the standard normal cdf;  Gini measure is introduced by \cite{Y82};  $X^{\prime}$ is an independent copy of $X$.
}
\end{table}

\section{Optimal Ordering Rule}

We now present an analytical solution to the worst-case risk minimization problem \eqref{wc-0}, offering insight into how optimal order quantities are shaped by three key elements: the cost structure--captured by the cost-to-price ratio $\beta = c/p$; demand uncertainty--measured by the coefficient of variation $r = \sigma/\mu$; and the newsvendor’s risk preference, encoded by the distortion function $h$.

\begin{theorem}\label{main-th}
Let $h \in \mathcal{H}_2 := \left\{ h \in \mathcal{H} : \int_0^1 (h'(u))^2\, \d u < \infty \right\}$, and define $s^* := h^{-1}(\beta)$, 
\[ 
\Delta_{s^*,t}:=\sqrt{t\int_{s^*}^{t} \left( h'(u)\right)^2\, \d u - (h(t)-\beta)^2},~t\ge s^* ~~\text{and}~~\sigma_t := \sqrt{t(\mu^2 + \sigma^2) - \mu^2},~  t \in \left[ \frac{1}{1 + r^2}, 1 \right].
\]
Then the optimal order quantity for problem \eqref{wc-0} satisfies one of the following:
\begin{itemize}[align=left]

\item [\textbf {(i) High demand uncertainty:}] 
If  $r \geq \sqrt{{1}/{s^*} - 1}$,
then $\rho^{\rm wc}_h(\mu,\sigma) =
0$, the optimal order quantity is $x^*=0$, and  the worst-case distribution is
 \begin{equation}\label{wosrt-dis-1}
 \p\left(S^*=\frac{\mu^2+\sigma^2}{\mu}\right)=\frac{1}{1+r^2}=1-\p(S^*=0).
 \end{equation}

\item [\textbf{(ii) Low demand uncertainty:}] 
If $r \leq \frac{\Delta_{s^*,1}}  { h'(1)-(1-\beta)}$, then   
$\rho_h^{\rm wc}(\mu,\sigma) =
-\mu(p-c) +  p \sigma  \Delta_{s^*,1} $,
the optimal order quantity is
  \begin{align} \label{eq:solutionx-2}
x^* =
\mu  - \sigma \frac{h^\tau(s^*)- 2(1-\beta)}{2\Delta_{s^*,1} },
\end{align}
where $h^\tau(s)$ is any value in $[h'(s),h'(s+)]$,
 and the worst-case distribution is
  \begin{align} \label{eq:wcdistributoin-2}
 F^{-1}(u)= 
\mu - \sigma \frac{h'(1-u)}{\Delta_{s^*,1}} 1_{\{u<1- s^*\}} + \sigma \frac{1-\beta}{\Delta_{s^*,1}}.
\end{align}

\item [\textbf {(iii) Intermediate demand uncertainty:}] 
\begin{sloppypar}If $\frac{\Delta_{s^*,1}}{ h'(1)-(1-\beta) }<r<\sqrt{ {1}/{s^*} - 1}$, 
then   
$\rho_h^{\rm wc}(\mu,\sigma) =
\frac{p}{t^*}\left[-\mu (h(t^*)-\beta) +\sigma_{t^*} \Delta_{s^*,t^*}\right]$,
 the optimal order quantity is\end{sloppypar}
\vspace{-\baselineskip}
  \begin{align*}
x^* =
   \frac{\mu}{t^*}  -   \frac{\sigma_{t^*}}{t^*}\frac{ t^*h^\tau(s^*)- 2(h(t^*)-\beta) }{2\Delta_{s^*,t^*}},
   \end{align*}
 and the worst-case distribution is
\begin{equation} \label{wwdis}
   F^{-1}(u) =  \begin{cases}
0, & u<1-t^*,\\
\frac{\mu}{t^*} - \frac{\sigma_{t^*}} {t^*} \frac{t^*h'(1-u)-(h(t^*)-(1-\beta))}{\Delta_{s^*,t^*}}, &   1-t^*\le u < 1-  s^*  ,\\
\frac{\mu}{t^*}+ \frac{\sigma_{t^*}}{t^*} \frac{h(t^*)-(1-\beta)}{\Delta_{s^*,t^*}}, &  u\ge 1-  s^* ,
\end{cases} 
\end{equation}
where 
\begin{align}
    \label{eq:tsatr}
  t^* 
 & =\max\left\{t\in  \left[\frac{1}{1+r^2},1\right]
 \left| (1+r^2)(th'(t)-h(t)+\beta)^2 - \int_0^1 (h'(u)-h (t)+\beta)^2\d u\le0 \right. \right\}.
\end{align}
 \end{itemize}
\end{theorem}

While more intricate than classical rules--such as those of \cite{S58}, \cite{NST11}, and \cite{HDZ14}--this added complexity reflects the model’s broader ability to represent virtually any risk-averse preference within the dual utility framework (\cite{Y87}).

Case {\bf (i)} offers a sharp ordering threshold: the newsvendor should refrain from ordering (i.e., set $x^* = 0$) whenever the coefficient of variation $r$ exceeds the critical level $\sqrt{1/s^* - 1}$, where $s^* = h^{-1}(\beta)$. For added intuition, this threshold condition is equivalent to $h\left( \frac{1}{1 + r^2} \right) \le \beta$, where $\frac{1}{1 + r^2} \in [0,1]$ can be interpreted as the implied probability of observing positive demand. The distortion function $h$ then adjusts this probability to reflect risk aversion. When the risk-adjusted likelihood of a positive sale falls below the cost-to-price ratio $\beta$, placing an order is no longer justified. This outcome arises when uncertainty is high (large $r$), costs are high (large $\beta$), or risk aversion is extreme—i.e., when $h$ approaches the worst-case risk measure ($\mathrm{ess\,sup}$).

Case {\bf (ii)} arises when the coefficient of variation $r$ is relatively small. The optimal order quantity takes on a familiar mean--standard deviation form:
$
x^* = \mu - k\sigma,\text{where}\; k := \frac{h^\tau(s^*) - 2(1 - \beta)}{2\Delta_{s^*,1}}.
$
The sign of $k$ is governed by the slope $h^\tau(s^*)$, which captures the decision maker’s sensitivity to downside risk near the threshold $s^* = h^{-1}(\beta)$. A steeper slope--corresponding to greater aversion to downside risk--yields a positive $k$, reducing the order quantity below the mean. Conversely, a flatter slope may result in a negative $k$, prompting a more aggressive order. 

Case {\bf (iii)} captures the intermediate regime between cases {\bf (i)} and {\bf (ii)}, where the coefficient of variation $r$ lies strictly between the two thresholds. The optimal order quantity now depends on an endogenous probability variable $t^* \in \left[\frac{1}{1 + r^2}, 1\right]$, representing the worst-case likelihood of positive demand (with $1 - t^*$ reflecting the probability of zero demand; see \eqref{wwdis}). The value of $t^*$ in \eqref{eq:tsatr} can be easily calculated via bisection,  leveraging the monotonicity of the left-hand side of the inequality.  When $h$ is piecewise linear, the optimal $t^*$ coincides with one of the slope breakpoints of the distortion function—yielding a closed-form solution in its simplest form. Notably, case {\bf (ii)} can be recovered as a limiting case when $t^* = 1$, corresponding to strictly positive demand.

\begin{corollary}\label{co:1}
If $h$ is piecewise linear, that is, there exist numbers
 $0 = a_0 < a_1 < \cdots < a_n = 1,$ $   n\ge 2,$ 
such that
 $(0,1] = \bigcup_{j=1}^n (a_{j-1},a_j],$ 
and $h$ is linear on each interval $(a_{j-1},a_j]$, 
then the optimal $t^*$ must be one of the $a_j$'s.
\end{corollary}
\begin{proof}[Proof.] 
By standard manipulation, one can verify that the inequality of \eqref{eq:tsatr} 
is equivalent to
 $$
w(t):=   (1+r^2) \left(\int_0^t(h_{s^*}'(t)-h_{s^*}'(u)\d u \right)^2- \int_0^1 (h_{s^*}'(u)-h_{s^*}(t))^2\d u\le0.
 $$
where $h_{s^*}(\cdot)=(h(\cdot)-h(s^*))_+/(1-\beta)$. Note that  $s^*\le 1/(1+r^2)$ and thus  $h_{s^*}'(t)=h_{s^*}'(u)$ for $\max\{ a_{j-1},1/(1+r^2))\}<u\le t\le a_{j}$ for any $j$. It follows that $w(t)$ takes constant value when $t\in(a_{j-1},a_j]$, and thus, $t^*$ must be one of the $a_j$'s.  

\end{proof}
\vspace{-\baselineskip}
\vspace{0.1in}
\paragraph{Worst-case distributions} 
The preceding discussion has focused on optimal order quantities across the three demand regimes. We now turn to the structure of the worst-case distributions that underpin these decisions in the inner risk evaluation problem. Unlike earlier models—such as Scarf’s risk-neutral formulation (\cite{S58}) and its risk-averse extensions under CVaR (\cite{NST11}) and mean–standard deviation (\cite{HDZ14})—whose worst-case distributions collapse to degenerate two-point forms, our distortion-based formulation yields significantly richer structures. Depending on the shape of the distortion derivative $h'$, the worst-case distribution can be discrete, continuous, or a mixture of both. 

At a high level, the worst-case distribution in \eqref{wwdis} consists of a point mass on high-demand realizations (exceeding the order quantity), along with a complementary component—discrete or continuous—spanning low-demand outcomes. This component distributes probability
across a spectrum of adverse scenarios, with the weighting governed by the perceived severity of risk encoded by $h$. When $h'$ is continuous, the worst-case distribution may itself be continuous.
This structural flexibility offers a more realistic and behaviorally interpretable view of how risk-averse decision makers hedge against uncertainty—moving beyond the stylized and overly conservative degeneracy imposed by earlier models.




\section{Impact of Risk Aversion on Optimal Ordering Behavior}

Building on the ordering rules from the previous section, we uncover a more nuanced and structured relationship between risk aversion and the optimal order quantity.
As a baseline, we begin by revisiting the CVaR model from \cite{NST11}, which arises as a special case of Theorem~\ref{main-th}.

\begin{example} [CVaR, \citealp{NST11}] For the distortion function $h(x)=(x-\alpha)_+/(1-\alpha)$ of {\rm CVaR}$_\alpha$, denote by $\eta= (1-\alpha)\left(1- c/p\right)$, and the optimal value and solution of  problem \eqref{wc-0} are respectively
 \begin{align*}
\rho_h^{\rm wc}(\mu,\sigma) = \begin{cases}
0,&\text{if}~\eta  \le \frac{\sigma^2}{\mu^2+\sigma^2} \\
(p-c)\left(-\mu+\sigma \sqrt{\frac{1-\eta}{\eta}}\right), & \text{if}~\eta  >\frac{\sigma^2}{\mu^2+\sigma^2}
\end{cases}
~~ {\rm and}~~
x^* = \begin{cases}
0, &\text{if}~\eta  \le \frac{\sigma^2}{\mu^2+\sigma^2} \\
\mu + \sigma \frac{ 2\eta-1}{2\sqrt{\eta(1-\eta)} }, & \text{if}~\eta  >\frac{\sigma^2}{\mu^2+\sigma^2}.
\end{cases}
\end{align*}
\end{example}

The solution always recommends ordering less as the risk aversion parameter $\alpha$ increases. While this behavior appears intuitive, it reflects the structural simplicity of the CVaR model and can be misleading. As we show below, the pattern already breaks down under a more realistic formulation—mean–CVaR. To pinpoint where this monotonicity holds or fails, we derive from Theorem~\ref{main-th} an explicit expression for the optimal order quantity $x^*$ across various parameter regimes.

\begin{example} 
[Mean-CVaR] The distortion function is
$h(u) = \lambda u + \frac{1-\lambda}{1-\alpha}(u-\alpha)_+$, with $\alpha \in [0,1]$ and $\lambda \in [0,1]$. We focus on cases {\bf (ii)} and {\bf (iii)}, where the optimal order is strictly positive.

\begin{enumerate} 
\item [(i)] If $\beta\ge \lambda\alpha$, then  
$s^*=h^{-1}(\beta) = \frac{\alpha(1-\lambda) +\beta(1-\alpha)}{1-\lambda\alpha}>\alpha$ and thus $t^*=1$. The optimal solution is
$$x^* = \mu  -  \sigma\frac{  \frac{1-\lambda\alpha}{(1-\alpha)(1-\beta)} - 2  }{2\sqrt{   \frac{1-\lambda\alpha}{(1-\alpha)(1-\beta)}    - 1}}.$$
 One can verify that   $x^*$ is decreasing in $\alpha$   and increasing in $\lambda$.
\item [(ii)]  If  $\beta < \lambda\alpha$ and   
$ 
\left(\beta-\frac{\lambda\alpha}{1-\alpha}\right)^{2} \leqslant \frac{\mu^{2}}{\sigma^{2}}\left(1- \beta {\lambda}+\frac{ (1-\lambda)^2\alpha}{1-\alpha}   - (1-\beta)^2\right),
$
then $s^*=h^{-1}(\beta)=\beta/\lambda < \alpha$ and $t^*=1$. 
The optimal solution is
  $$
x^* =
  \mu -  \sigma\frac{ \lambda- 2(1-\beta) }{2\sqrt{1- \beta {\lambda}+\frac{ (1-\lambda)^2\alpha}{1-\alpha}   - (1-\beta)^2}}.
$$
  One can verify that $x^*$ is increasing in $\alpha$ if $\lambda>2(1-\beta)$ and decreasing in $\alpha$ if $\lambda\le2(1-\beta)$. However, $x^*$ is not monotonic in $\lambda$ in general.
\item  [(iii)] If $\beta < \lambda\alpha$ and $ 
\left(\beta-\frac{\lambda\alpha}{1-\alpha}\right)^{2} > \frac{\mu^{2}}{\sigma^{2}}\left(1- \beta {\lambda}+\frac{ (1-\lambda)^2\alpha}{1-\alpha}   - (1-\beta)^2\right)
$, then $s^*=h^{-1}(\beta)<\alpha$ and $t^*=\alpha$. The optimal solution is
$$
x^* = \frac{\mu}{\alpha} + \frac{\sigma_{\alpha}}{2\alpha} \left[  \sqrt{\frac{   \alpha\lambda - \beta}\beta  } - 
   \sqrt{\frac{ \beta}{  \alpha\lambda - \beta }} \right].
$$
In this case, $x^*$ is increasing in $\lambda$, while its dependence on $\alpha$ is non-monotonic.
\end{enumerate}
\end{example}

The key takeaway is that the solution aligns with the CVaR model’s recommendation—ordering less as risk aversion increases (higher $\alpha$, lower $\lambda$)—only when the cost-to-price ratio $\beta$ is sufficiently high (i.e., $\beta \ge \lambda\alpha$). However, once this condition fails, the optimal order may increase with risk aversion, underscoring that monotonicity is not universal and that more risk-averse decision makers may rationally order more under distributional ambiguity when ordering is relatively inexpensive.

While the CVaR-based newsvendor model (\citealp{NST11}) may be overly simplistic, the mean–standard deviation model of \cite{HDZ14} raises a deeper concern—paradoxically recommending higher order quantities for more risk-averse decision makers when ordering is expensive (i.e., when $\beta$ is high). As a counterpoint, we examine the mean–deviation-from-median model, a coherent alternative within the mean–deviation family. The derivation is compressed and presented inline due to space constraints, but it follows the same steps as earlier examples.



\begin{example}[Deviation from the median] The distortion function is  
$$h'(u)=\left\{\begin{array}{ll} 1-a , & u<\frac{1}{2} ,\\   1+a , & u \geqslant \frac{1}{2},\end{array}\right.
 ~~{\rm and}~~s^*=\begin{cases}
 \frac \beta{  1-a }, & \beta \le \frac12(1-a),\\
 \frac {a+\beta}{1+a}, & \beta > \frac12(1-a).
 \end{cases}
$$
We focus on cases {\bf (ii)} and {\bf (iii)}, where the optimal order is strictly positive and the condition $h(\mu^2/(\mu^2+\sigma^2))> \beta$ holds--that is, $(1-a)\frac{\mu^2}{\mu^2+\sigma^2} > \beta$ if $\mu\le \sigma$ or
$(1+a)\frac{\mu^2}{\mu^2+\sigma^2} -a>\beta$ if $\mu> \sigma$. 
\begin{enumerate}
\item [(i)] If  $\beta \geq (1-a)/2$, 
  then $s^*=\frac {a+\beta}{1+a}\ge 1/2$ and thus $t^*=1$.  The optimal solution is
$$
x^*=\mu - \sigma \frac{ (1+a)/(1-\beta)-2}{2\sqrt{ (1+a) / (1-\beta) -1} },
$$
which is decreasing in $a$.
\item [(ii)]  If   $\beta < (1-a)/2$ and 
    $(a+\beta)\sigma/\mu \le \sqrt{a^2 + \beta a +\beta(1-\beta)}$, 
     then $s^*=\frac \beta{(1-a)}<1/2$ 
 and thus $t^*=1$. The optimal solution is
$$
x^* = \mu - \sigma \frac{-a-1+2 \beta }{2\sqrt{a^2+a\beta+\beta(1-\beta)} }
$$
which is increasing in $a$ for $\beta>1/3$ and not monotone in $a$ for $\beta<1/3$ in general.  
\item  [(iii)] If   $\beta < (1-a)/2$ and  $(a+\beta)\sigma/\mu > \sqrt{a^2 + \beta a +\beta(1-\beta)}$, then we have $t^*=1/2$. The optimal solution is
$$x^* = 2{\mu}+{\sqrt{\frac{\sigma^2-\mu^2}2}} \times \frac{c_0-2} {\sqrt{c_0-1}}$$ with $c_0:=\frac{1-a}{2\beta}$
which is decreasing in $a$.
\end{enumerate}
\end{example}

In sharp contrast to the mean–standard deviation model of \cite{HDZ14}, the above solution recommends ordering less as risk aversion ($a$) increases when ordering is expensive (i.e., $\beta \geq (1-a)/2$), thereby resolving the paradox and aligning with practical intuition. When ordering is inexpensive, however—similar to the mean–CVaR case—more risk-averse newsvendors may rationally order more. To illustrate this consistent behavioral pattern, Figure~\ref{fig:distortion_grid} shows how the optimal order quantity responds to increasing risk aversion across varying values of $\beta$ and the coefficient of variation $r$, spanning the major distortion families introduced earlier. A consistent pattern emerges under high $\beta$: the optimal order decreases with risk aversion. In contrast, under low $\beta$, the response can be non-monotonic.

\newcommand{\rowlabel}[2][0pt]{%
  \raisebox{#1}{\rotatebox{90}{\scriptsize #2}}}

\newcommand{\ylabel}[1][0pt]{%
  \raisebox{#1}{\rotatebox{90}{\textbf{Optimal order quantity}}}}

\newcommand{\xlabel}[1][0pt]{%
  \hspace*{#1}\textbf{Risk aversion}\hspace*{0pt}}

\begin{figure}[htbp]
  \centering
  \setlength{\tabcolsep}{3pt}

  \begin{tabular}{@{}c@{\hspace{6pt}}c@{}}
    \ylabel[-60pt] &                                   
    \begin{tabular}{@{}c c c c@{}}
        & \textbf{$\beta=0.25$} & \textbf{$\beta=0.556$} & \textbf{$\beta=0.7$} \\[.25em]

        \rowlabel[2.5em]{Mean-CVaR} &
        \includegraphics[width=.30\textwidth]{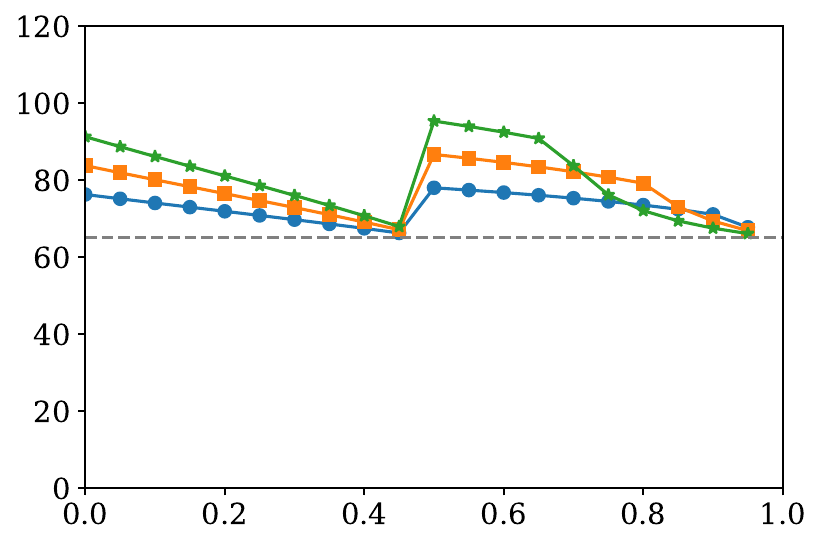} &
        \includegraphics[width=.30\textwidth]{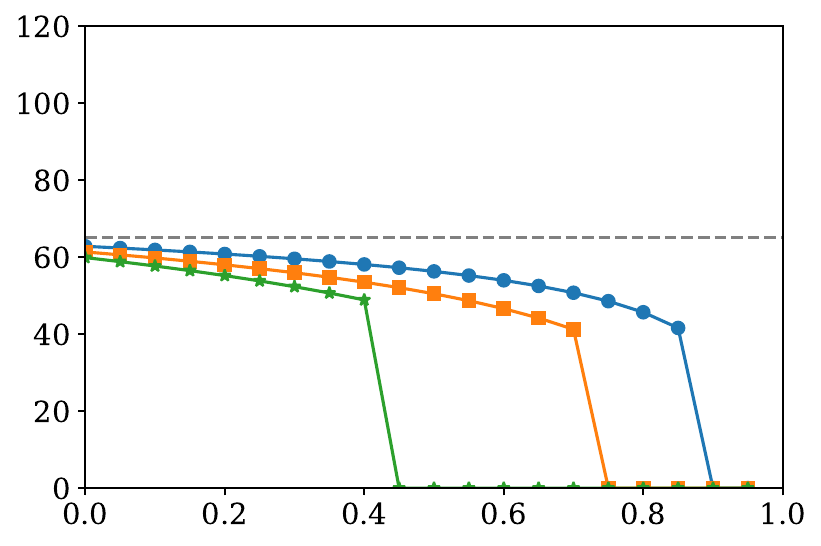} &
        \includegraphics[width=.30\textwidth]{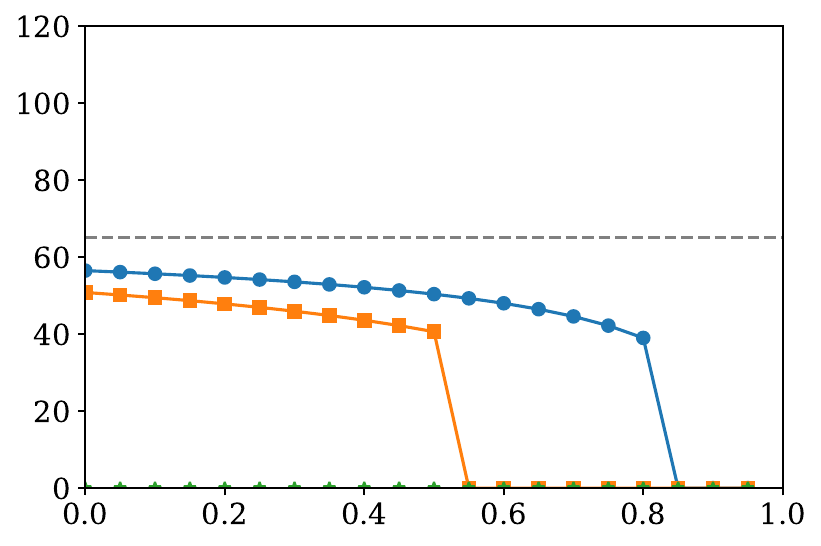} \\[.3em]

        \rowlabel[0.95em]{Deviation from Median} &
        \includegraphics[width=.30\textwidth]{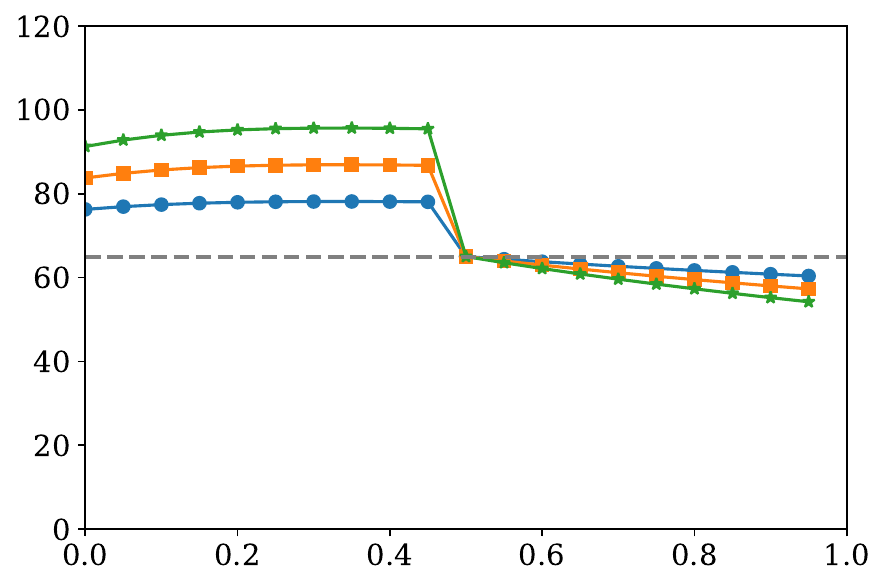} &
        \includegraphics[width=.30\textwidth]{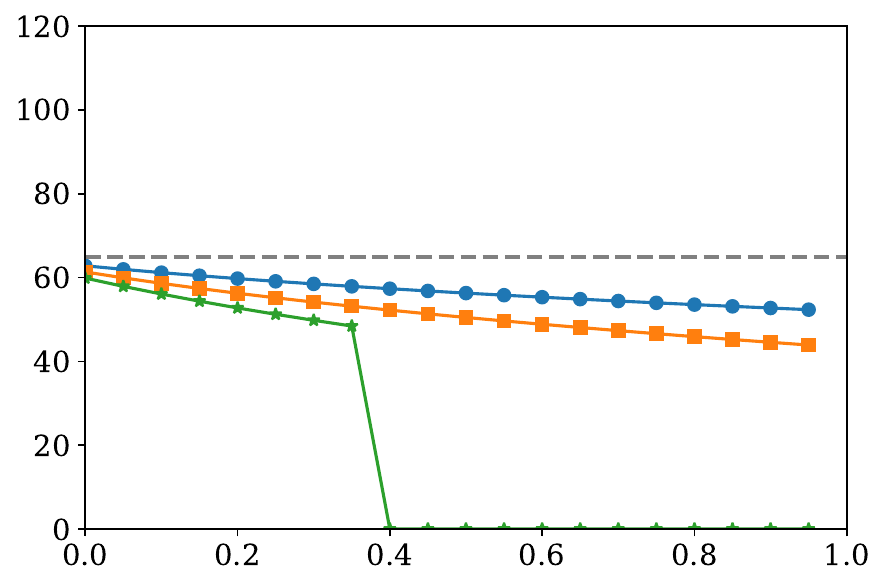} &
        \includegraphics[width=.30\textwidth]{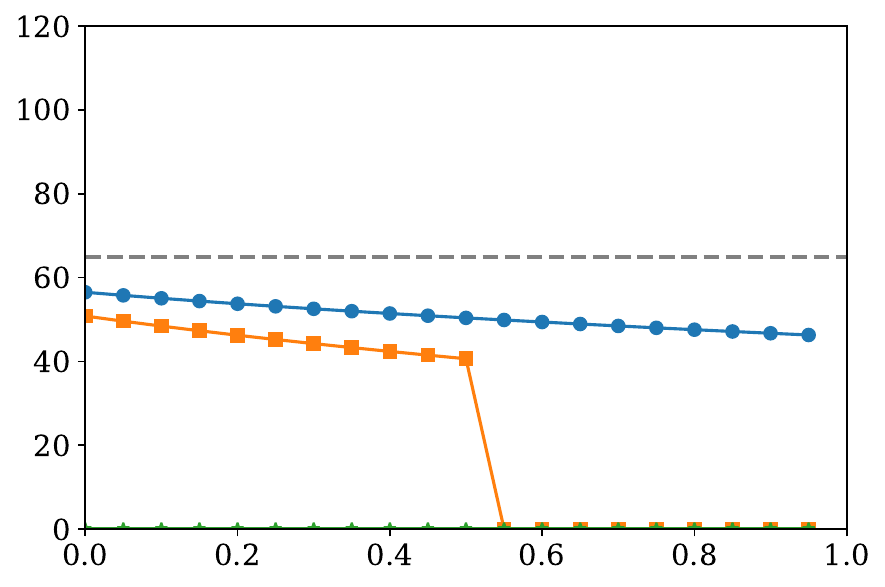} \\[.3em]

        \rowlabel[1.9em]{Wang Transform} &
        \includegraphics[width=.30\textwidth]{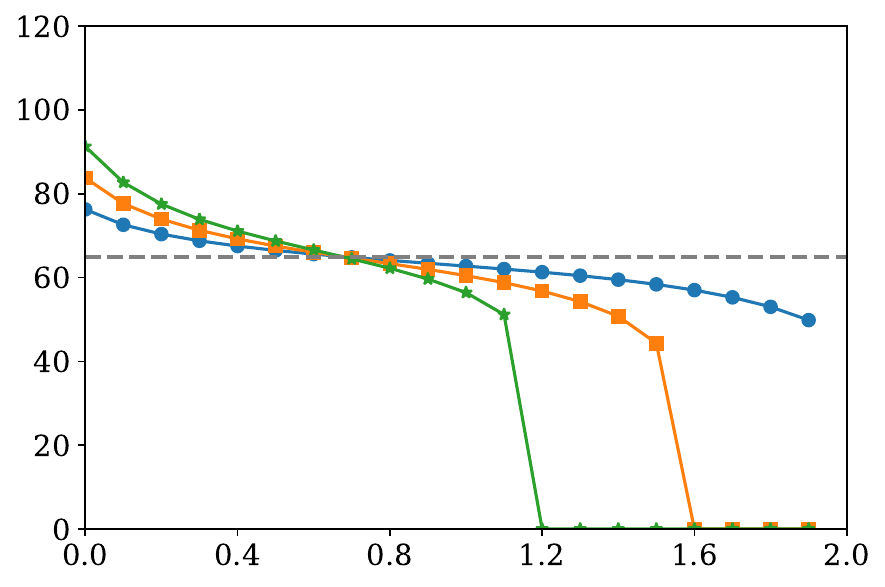} &
        \includegraphics[width=.30\textwidth]{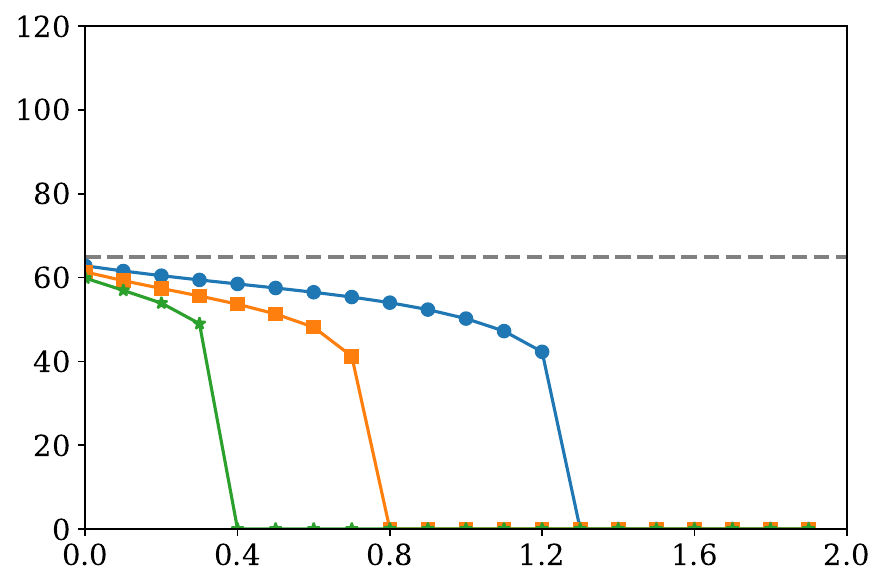} &
        \includegraphics[width=.30\textwidth]{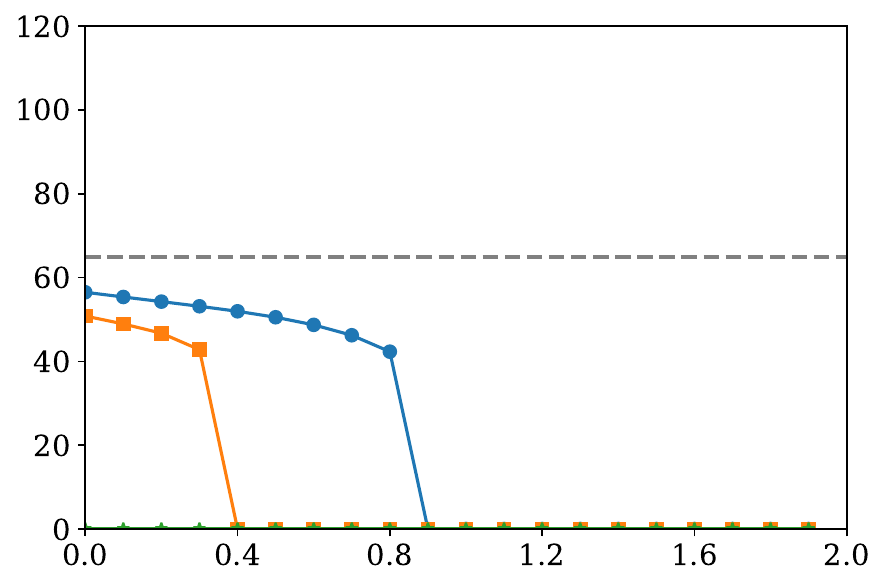} \\[.3em]
        
        \rowlabel[1.4em]{Proportional Hazards} &
        \includegraphics[width=.30\textwidth]{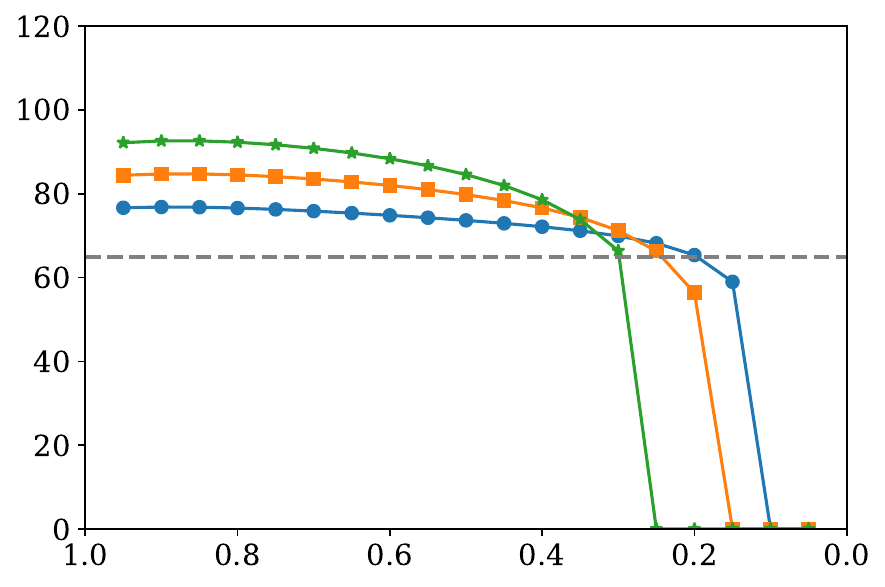} &
        \includegraphics[width=.30\textwidth]{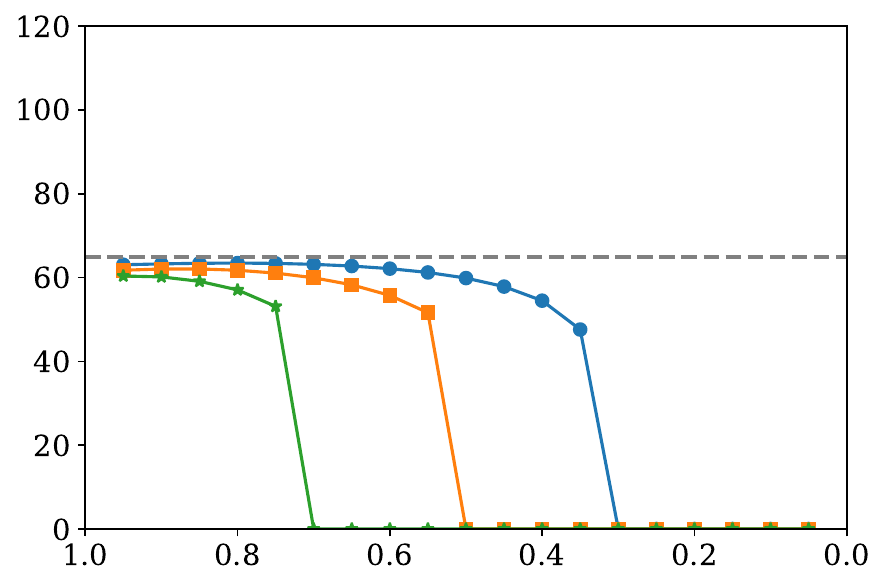} &
        \includegraphics[width=.30\textwidth]{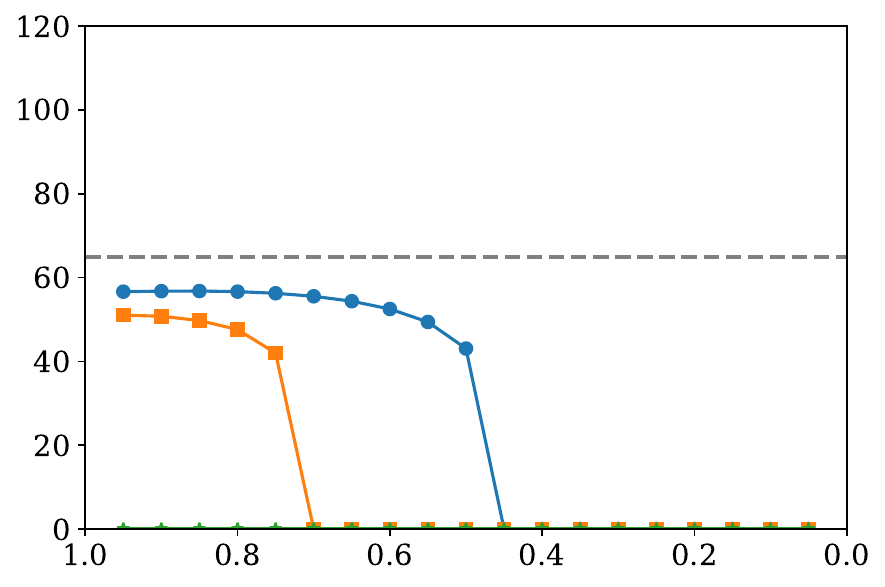} \\[.3em]

        \rowlabel[2.3em]{Gini Measure} &
        \includegraphics[width=.30\textwidth]{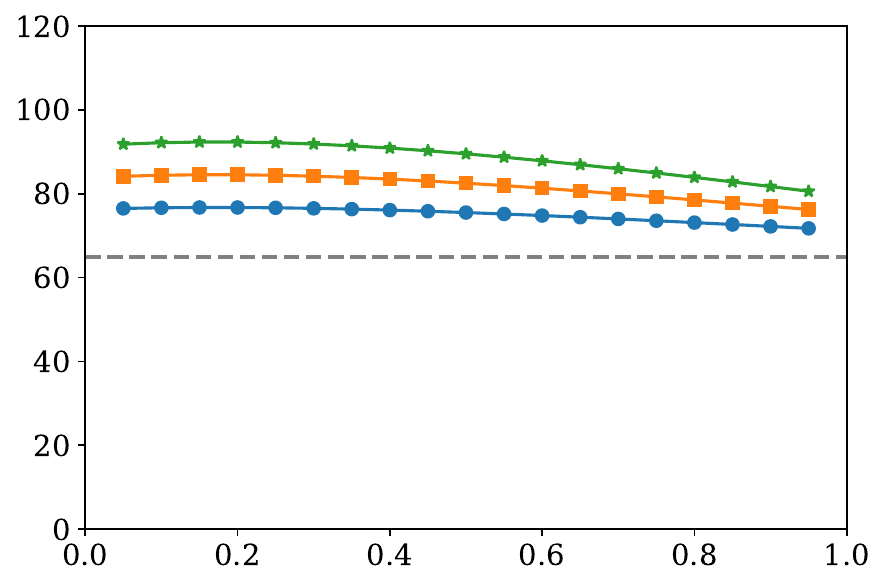} &
        \includegraphics[width=.30\textwidth]{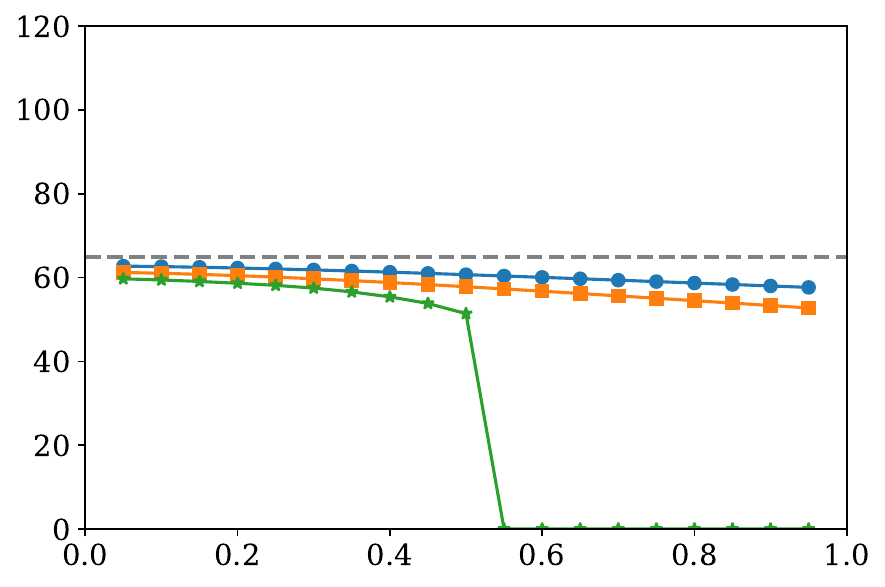} &
        \includegraphics[width=.30\textwidth]{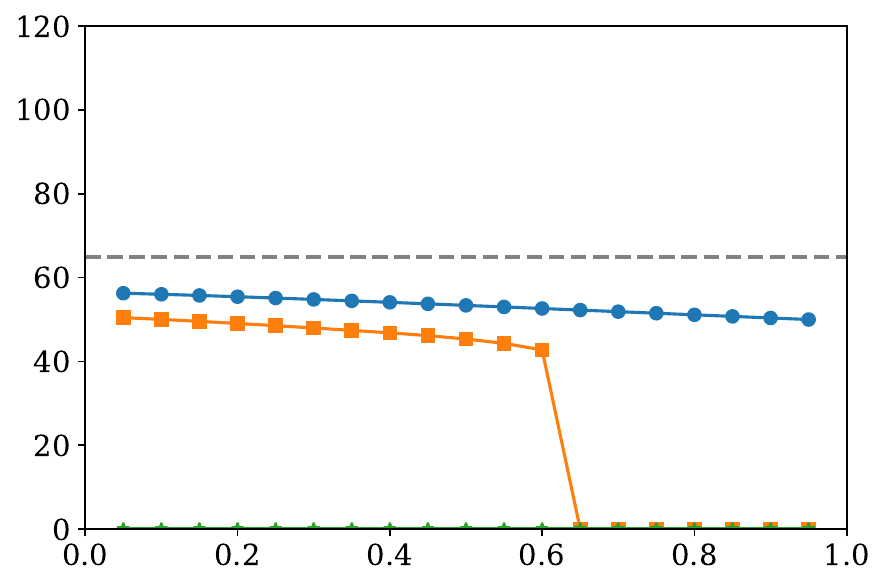} \\

    \end{tabular}
  \end{tabular}


  \includegraphics[height=0.5cm]{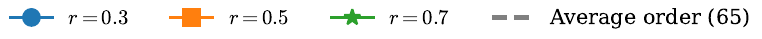}   


  \xlabel[4em]     

  \caption{ Optimal order quantity (vertical axis) versus risk-aversion parameter (horizontal axis)
    under the five risk measures from Section 2—%
    \textbf{Mean-CVaR} (with $\lambda = 0.5$ and $\alpha \in [0,1]$), \textbf{Deviation from Median},
    \textbf{Wang Transfer}, \textbf{Proportional Hazard} and \textbf{Gini Measure}.
    Columns correspond to the cost-to-price ratios $\beta\in\{0.25,\,0.556,\,0.7\}$,
    while marker shape and colour encode the coefficient of variation
    $r\in\{0.3,\,0.5,\,0.7\}$ (see legend).
    The grey dashed line marks the average–order benchmark of 65.
    \textit{Note: for the Proportional-Hazard the horizontal axis is
  $1\!-\!a$ and is plotted in reverse (right-to-left) so that increasing
  risk aversion aligns directionally with the other panels.}}
  \label{fig:distortion_grid}
\end{figure}


\section{Proofs of Theorem \ref{main-th}}
We first consider the inner problem of \eqref{wc-0}: 
\begin{align}
    \label{ws}
   \sup_{F \in \mathcal{S}(\mu, \sigma)} \rho^F_h(\psi(x, p, c)) =   p \sup_{F \in \mathcal{S}(\mu, \sigma)} \rho^F_h(\max\{-S, -x\}) +c x.
\end{align}
For $s\le t$, $k\ge 0$ and $\ell\in\R$, define a distribution   $G_{s,t}$ such that its quantile function is
\begin{equation}\label{eq-20122-2}
 G_{s,t}^{-1}(u) =
  \ell1_{\{u\le s\}} + (k h'(u)+\ell) 1_{\{s<u\le t\}}  ~~{\rm with}~~k h'(t)+\ell\le 0.
\end{equation}
\begin{lemma} \label{lm-worstdf}
For $x\ge 0$ and $h\in \mathcal H_2$, the  worst case distribution of $-S$ of  the problem \eqref{ws}   has the form of $G_{s,t}$ whose inverse function is given by \eqref{eq-20122-2} with $k\ge 0$, $\ell\in\R$, $s\le t$  such that $\ell\le -x$ if $s>0$ and $k h'(u)+\ell\ge -x$ for $u>s$ if $s<t$. 
\end{lemma}
 \begin{proof}[Proof.]
We first consider the case that $h$ is piecewise linear, that is,   there exist $0=a_0<a_1<\cdots<a_n=1$, $n\ge 2$, such that $(0,1]=\cup_{j=1}^n (a_{j-1},a_{j}]$ and $h$ is linear on $(a_{j-1},a_{j}]$ with slope $\lambda_j$, $j\in[n]$. Denote by 
$\delta_j=a_j-a_{j-1}$, $j\in[n]$. 
We assert that in this case, the worst-case distribution $G$ of $-S$ of  the problem \eqref{ws}   has the form of
\begin{align} \label{eq:wccdf-discrete}
G &=  [-x_0,s;-x_i,a_i - s;-x_{i+1}, \delta_{i+1}; \ldots; -x_{j},\delta_{j}; 0,1-a_j],
\end{align}
where $0\le x_j\le x_{j-1}\le \cdots\le x_i\le x\le x_0,$  $s\in [0,a_i)$ and   $i,\,j\in [n]$ and $[x_1,p_1;\ldots,x_n,p_n]$ denotes the distribution that assigns probability $p_i$ to $x_i$ for each $i \in [n]$. 
 It suffices to show for any $F \in \cal S(\mu,\sigma)$, there exists $F^*\in \cal S(\mu,\sigma)$ having the form of \eqref{eq:wccdf-discrete} such that
 \begin{equation}
\label{eq-20-229-1}
\rho_h^F\big(\max\{-S,-x\}\big)\le \rho_h^{F^*}\big(\max\{-S,-x\}\big).
\end{equation} Let $S\sim F$ and $U\sim {\rm U}[0,1]$ such that $U$ and $-S$ are comonotonic. Let $s:=\p(-S\le -x)$ 
and  $i_s\in [n]$ be such that $a_{i_s-1}\le s<a_{i_s}$. We   consider  the following three cases.
\begin{enumerate}
\item [(i)] If $s\in (0,1),$ let
$A_0=\{-S\le -x\}$, $A_{i_s} = \{a_{{i_s}-1}<U \le a_{i_s}, -S>-x\}$,  $A_{i} = \{a_{i-1}<U \le a_i\}$, $i=i_s+1,\ldots,n$, and for $\boldsymbol{\varepsilon}=(\varepsilon_{i_s},\ldots,\varepsilon_n)\in \prod_{j=i_s}^n[0, x_j]$, define  
$ \overline{S}_{\boldsymbol{\varepsilon}}
  =: (x_0+\overline{\varepsilon})\id_{A_0} + \sum_{j=i_s}^{n} (x_j-\delta_j)_+\id_{A_j},$
where  $x_i=\E[X|A_i]\ge 0,$ $i\in\{0,i_s,\ldots,n\},$ and $\overline{\varepsilon}\ge 0$ is a constant such that $\E[ \overline{S}_{\boldsymbol{\varepsilon}}]= \mu$.   One can verify  ${\rm Var}(\overline{S}_{\bf 0})\le \sigma^2$ and ${\rm Var}( \overline{S}_{\boldsymbol{\varepsilon}})\ge {\rm Var}(\overline{S}_{\bf 0})$ and ${\rm Var}( \overline{S}_{\boldsymbol{\varepsilon}})$ is continuous and strictly increasing (componentwise) in $\boldsymbol{\varepsilon}\in \prod_{j=i_s}^n[0, x_j]$.  
 If  $
{\rm Var}(\overline{S}_{\boldsymbol{\varepsilon}_x})<\sigma^2$ with $\boldsymbol{\varepsilon}_x=(x_{i_s},\ldots,x_n)$, that is,
$ 
{\rm Var}\left(\frac{\mu}{s}1_{\{U \le s\}}\right) <\sigma^2,
$ 
then there exists $s^*\in (0,s)$ such  that   ${\rm Var}(S^*)=\sigma^2$ where $S^*:=\frac{\mu}{s^*}1_{\{U \le s^*\}}$. Since $\frac{\mu}{s}1_{\{U \le s\}}\preceq_{\rm cx} S^*$,\footnote{For two random variables $X$ and $Y$, $X$ is said to be smaller than $Y$ in convex order, denoted by $X\preceq_{\rm cx}Y$, if $\E[\phi(X)]\le \E[\phi(Y)]$ for all convex functions $\phi$  for which the expectations exist. We refer the reader to \cite{SS07} for further details.} we have $\overline{S}_{\bf 0}\preceq_{\rm cx} S^*$.
If $
{\rm Var}(\overline{S}_{\boldsymbol{\varepsilon}_x})\ge \sigma^2$, 
  then we can find $\boldsymbol{\varepsilon}^*\in  \prod_{j=i_s}^n[0, x_j]$ such that  ${\rm Var}(S^*)=\sigma^2 $ where $S^*:=\overline{S}_{\boldsymbol{\varepsilon}^*}\succeq_{\rm cx}\overline{S}_{\bf 0}$. 
In both cases, we find a random variable $S^*$ with mean $\mu$ and variance $\sigma^2$ such that $\overline{S}_{\bf 0}\preceq_{\rm cx} S^*$ which implies $\max\{-\overline{S}_{\bf 0},-x\}\prec_{\rm icx}\max\{-S^*,-x\}$, and thus, $\rho_h^F\big(\max\{-S,-x\}\big)=\rho_h^F\big(\max\{-\overline{S}_{\bf 0},-x\}\big)\le \rho_h^{F^*}\big(\max\{-S^*,-x\}\big)$, that is, \eqref{eq-20-229-1} holds,
where $F^*$ is the distribution of $S^*$. Observing the structure of  $F^*$,  we get the desired $F^*$.

\item [(ii)] If $s=0$, that is, $-S>-x$ a.s., then letting $A_0=\{U\le a_1\}$, $A_i = \{a_{j-1}<U \le a_j\}$, $i=2,\ldots,n$, and by  similar arguments as in Case (i), we can construct $F^*\in \mathcal {S}(\mu,\sigma)$ such that \eqref{eq-20-229-1} holds.

\item [(iii)]If $s=1$, then $-S\le -x$ a.s. and thus,  $\rho_h^F(\max\{-S,-x\})= \inf_{F\in \cal S(\mu,\sigma)} \rho_h^F\big(\max\{-S,-x\}\big).$ 
It follows that any $F^*\in \mathcal {S}(\mu,\sigma)$ having the form of \eqref{eq:wccdf-discrete} is the desired distribution.
\end{enumerate}
Combining the above three cases, it suffices to consider the distribution  with the form of \eqref{eq:wccdf-discrete}.  That is, 
$\sup_{F\in \mathcal S(\mu,\sigma)} \rho_h^F(\psi(x,p,c))=\sup_{F\in \mathcal S_h(\mu,\sigma)} \rho_h^F(\psi(x,p,c))$, where  $\mathcal S_h(\mu,\sigma)$ is the set of all the distributions in $\mathcal S(\mu,\sigma)$ with the form of \eqref{eq:wccdf-discrete}.  Also by the convexity of $\rho_h$ and $s\mapsto\max\{-s,-x\}$,  we have that $\sup_{F\in \mathcal S(\mu,\sigma)} \rho_h^F(\psi(x,p,c))=\sup_{F\in \overline{\mathcal S}(\mu,\sigma)} \rho_h^F(\psi(x,p,c))$, where $\overline{\mathcal S}(\mu,\sigma)=\bigcup_{t\le \sigma}\mathcal S(\mu,t)$, and thus, 
$$
\sup_{F\in \mathcal S_h(\mu,\sigma)} \rho_h^F(\psi(x,p,c))=\!\!\sup_{F\in \overline{\mathcal S}_h(\mu,\sigma)} \rho_h^F(\psi(x,p,c))=\!\!\sup_{F\in \mathcal S(\mu,\sigma)} \rho_h^F(\psi(x,p,c))=\!\!\sup_{F\in \overline{\mathcal S}(\mu,\sigma)} \rho_h^F(\psi(x,p,c)),
$$
where $\overline{\mathcal S}_h(\mu,\sigma)=\bigcup_{t\le \sigma}\mathcal S_h(\mu,t)$ lies between $\mathcal S_h(\mu,\sigma)$ and $\overline{\mathcal S}(\mu,\sigma)$.
Therefore,  problem \eqref{ws} is equivalent to
  \begin{align}
  \sup_{s\in [0,1],\,1\le i\le j\le n}~~\sup_{ (x_0,x_i,\ldots,x_j) } ~&  -h(s) x - \delta_i^* \lambda_i x_i -\sum_{t=i+1}^j \lambda_t \delta_t x_t 
  \label{eq-20229-5}\\
 {\rm subject ~to} ~~&s x_0+\delta_i^* x_i + \sum_{t=i+1}^j \delta_t x_t  =\mu,\nonumber\\
  &s x_0^2 +\delta_i^* x_i^2 + \sum_{t=i+1}^j \delta_t x_t^2  \le \mu^2+\sigma^2 \nonumber\\
   & -x_0\le -x\le -x_i  \le \cdots\le-x_{j-1}\le  -x_j\le 0
   \nonumber,
\end{align}
\begin{sloppypar}
    
in the sense that the worst case distribution of $-S$ is $[-x_0^*,s^*;-x_{i^*},\delta_{i^*}^*;\ldots;-x_{j^*},\delta_{j^*};0,1-a_{j^*}]$ if $(s^*,i^*,j^*,x_0^*,x^*_{i^*},\ldots,x^*_{j^*})$ is the optimal solution to problem \eqref{eq-20229-5}. 
 Given $s\in [0,1]$ and $j\ge i$, consider the inner  problem of \eqref{eq-20229-5} which is a convex program. 
 By considering  its dual problem, we have its solution  satisfies  $-x_0 = \min\{\ell,-x\}$ and $-x_t=(k \lambda_{k} +\ell)_{-x}^0$, $t=i,\ldots,j$, for some $k\ge 0$ and $\ell$, where  $(a)_x^0= \min\{\max\{a,x\},0\}$. 
Noting that  $k \lambda_{t} +\ell$ is increasing in $t$, if $k\lambda_j+\ell>0$, then  $(k\lambda_j+\ell)_{-x}^0=0$ and we can set $j$ to be a smaller one.  
This implies that    problem   \eqref{eq-20229-5}  is equivalent to \end{sloppypar} \vspace{-\baselineskip}
  \begin{align} \label{eq-20403-1}
  \sup_{s,\,j\ge i,\,k\ge 0,\,\ell } ~&-h(s) x  + \delta_i^* \lambda_i (k \lambda_{i} +\ell)_{-x} +\sum_{t=i+1}^j \lambda_t \delta_t (k\lambda_t+\ell)_{-x}
 \\
 {\rm subject ~to}  ~~& s (\ell)^{-x}+\delta_i^*  (k \lambda_{i} +\ell)_{-x}  + \sum_{t=i+1}^j \delta_t (k\lambda_t+\ell)_{-x}  =\mu,~~~ k\lambda_j+\ell\le 0,\nonumber\\
  &s ((\ell)^{-x})^2 +\delta_i^* ((k \lambda_{i} +\ell)_{-x} )^2 + \sum_{t=i+1}^j \delta_t ((k\lambda_t+\ell)_{-x})^2  \le \mu^2+\sigma^2, \nonumber 
\end{align}
\begin{sloppypar}
    
in the sense that the optimal solution of problem  \eqref{eq-20229-5} is $(s^*,i^*,j^*, \min\{\ell^*,-x\},\max\{k^*\lambda_{i^*}+\ell^*,-x\},\ldots,\max\{k^*\lambda_{j^*}+\ell^*,-x\})$, and thus, the worst case distribution of $-S$ is \end{sloppypar}\vspace{-\baselineskip}
\begin{equation} \label{wc-distribution2}[\min\{\ell^*,-x\},s^*;\max\{k^*\lambda_{i^*}+\ell^*,-x\},\delta_{i^*}^*;\ldots;\max\{k^*\lambda_{j^*}+\ell^*,-x\},\delta_{j^*};0,1-a_{j^*}],
\end{equation}
if $(s^*, i^*,j^*, k^*,\ell^*)$ is the optimal solution to problem \eqref{eq-20403-1}. 
Comparing \eqref{wc-distribution2} and \eqref{eq:wccdf-discrete}, we conclude that $\ell^*\le -x$ if $s^*>0$ and $-x\le k^*\lambda_{i^*}+\ell^*$ if $s^*<t^*$ which completes the proof. To see this, first note that  if $\ell^*> -x$, then  $ 
       \max_{F\in\mathcal {S}(\mu,\sigma)} \rho_h^F(-S\vee -x) =\max_{F\in\mathcal {S}(\mu,\sigma)} \rho_h^F(-S).
    $ 
Note that the worst-case distribution of  $-S$ to the problem $\max_{F\in\mathcal {S}(\mu,\sigma)} \rho_h^F(-S)$ is
    $[k \lambda_{1} +\ell,\delta_1;\ldots;k\lambda_j+\ell,\delta_j;0,1-a_j]$ 
for some $k\ge 0$ and $\ell$. This is   included in \eqref{eq-20122-2} with  $s=0$.
Now suppose that $k^*\lambda_{i^*}+\ell^*<-x$. If $s^*>0$,
then \eqref{wc-distribution2} reduces to 
 $[\ell^*,s^*;-x,\delta_{i^*}^*;\ldots;\max\{k^*\lambda_{j^*}+\ell^*,-x\},\delta_{j^*};0,1-a_{j^*}]$. By case (i), we know this is not an optimal distribution to the problem \eqref{ws}.
  If $s^*=0$, then \eqref{wc-distribution2} reduces to 
 $[-x,\delta_{1};\ldots;\max\{k^*\lambda_{j^*}+\ell^*,-x\},\delta_{j^*};0,1-a_{j^*}]$ which is a feasible distribution of $-S$ to the problem $\sup_{F\in \mathcal S(\mu,\sigma)}\rho_h^F(-S)$ whose optimal distribution is  $[k \lambda_{1} +\ell,\delta_1;\ldots;k\lambda_j+\ell,\delta_j;0,1-a_j]$ 
for some $k\ge 0$ and $\ell$. This is included in \eqref{eq-20122-2} with  $s=0$. 
This completes the proof of the case that $h$ is piecewise linear.

For $h\in \mathcal H_2$, there exists  a sequence of piecewise linear convex distortion functions $h_n\in \mathcal H_2$, $n\in\N$, such that $h_n\uparrow h$ and $\int_0^1 (h_n'(u) - h'(u))^2\d u \to 0$ as $n\to\infty$.  Let $h_0=h$ and for each $n\ge0$, denote by $\mathcal G_n$ the set of all distributions of $G_{s,t}$ whose inverse function is given by \eqref{eq-20122-2} with $h=h_n$,  $k\ge 0$, $\ell\in\R$, and $s\le t$. By the above result for piecewise linear case,  we have  
\begin{equation}\label{eq-201016-1}
\sup_{F\in  \mathcal G_n} \rho_{h_n}^F\big(\psi(x,p,c)\big)=\sup_{F\in \cal S(\mu,\sigma)} \rho_{h_n}^F\big(\psi(x,p,c)\big).  
\end{equation}
Note that for each $F\in \mathcal S(\mu,\sigma)$, we have $\E^F[(\psi(x,p,c))^2]\le 4p^2(\mu^2+\sigma^2)+4(p^2+c^2)x^2=:C$.
 It then follows that $  
 |\rho_{h_n}^F(\psi) - \rho_{h}^F(\psi)|^2\le C\int_0^1 (h_n'(u) - h'(u))^2\d u \to 0 $ as  $n\to\infty$ and thus, 
\begin{equation}\label{eq-201016-3}
\lim_{n\to\infty} \sup_{F\in \cal S(\mu,\sigma)} \rho_{h_n}^F\big(\psi(x,p,c)\big)  =  \sup_{F\in \cal S(\mu,\sigma)} \rho_{h}^F\big(\psi(x,p,c)\big).
\end{equation}
On the other hand,  for each $ G\in \mathcal G_0$, there exists a sequence of distributions $G_n\in\mathcal G_n $, $n\in\N$, such that 
$
 \int_0^1 (G_n^{-1}(u)-G^{-1}(u))^2\d u \to0,$ $n\to\infty.
$
This implies that
$ 
 |\rho_{h_n}^{G_n}(\psi) - \rho_{h}^G(\psi)|^2   \to 0$  as $ n\to\infty,
$  
and thus,
\begin{equation}\label{eq-201019-2}
\lim_{n\to\infty}\sup_{F\in  \mathcal G_n} \rho_{h_n}^F\big(\psi(x,p,c)\big)=  \sup_{F\in  \mathcal G_0} \rho_{h}^F\big(\psi(x,p,c)\big).\end{equation}
Combining \eqref{eq-201016-1}, \eqref{eq-201016-3}, and \eqref{eq-201019-2}, we complete the proof for general distortion functions. 
\end{proof}
 
The next lemma is used to give the explicit form of feasible  $G_{s,t}$ defined by \eqref{eq-20122-2} to the problem~\eqref{ws}.  Note that we can write $G_{s,t}^{-1}(u) = (\overline{k}h_s'(u)+\ell)1_{\{u\le t\}}$ with  $\overline{k}h_s'(t)+\ell\le 0$, where  $$
 h_s(u) =   \frac{h(u)-h(s)}{1-h(s)}\id_{\{s<u\le 1\}},~~~~u\in [0,1]. 
$$  
 \begin{lemma}\label{lm-31}
Given $h\in\mathcal H_2$ and $0\le s\le t\le 1$,    there exists   $(k,\ell)\in\R_+\times\R$ such that
 \begin{align} \label{eq-0409-1}
   \int_0^t k  h_s'(u)+\ell\, \d u =-\mu ~~~{\rm and}~~~
   \int_0^t \left(k h'_s(u)+\ell\right)^2\, \d u  =\mu^2+\sigma^2 
 \end{align}
 only if   $t\ge \frac{1}{1+r^2}$. Moreover,  when $t= \frac{1}{1+r^2} $, the   solution  
 is $k=k_{s,t}:=0$ and $\ell=\ell_{s,t}:=-\mu/t$, and when $t> \frac{1}{1+r^2}$, if a solution exists ($h'_s\not\equiv0$ on $[0,t]$), then the  solution  
 is
 \begin{align}\label{eq-klst}
 k= k_{s,t}: =\frac{\sigma_t}{\sqrt{t\int_0^t \left( h'_s(u)\right)^2\, \d u - (h_s(t))^2}},~~~
   \ell=\ell_{s,t} := -\frac{\mu }{t} - k \frac{h_s(t)}t.
 \end{align}
 \end{lemma}
\begin{proof}[Proof.]
  Note that  \eqref{eq-0409-1}  
implies
 $
  k^2 (t\int_0^t \left( h_s'(u)\right)^2 \d u - h_s^2(t) ) =t (\sigma^2+\mu^2) - \mu^2.
 $
By the Cauchy-Schwarz inequality, we have $t\int_0^t \left( h_s'(u)\right)^2 \d u - h_s^2(t)\ge 0$ and the equality holds if $h_s'$ is a constant on $ [0,t]$.  This equation has a solution only if $t\ge \frac{1}{1+r^2}$. 
The solution can be checked by standard manipulation.
 Thus, we complete the proof.  
\end{proof}

By Lemmas \ref{lm-worstdf} and \ref{lm-31},
the worst case distribution of problem \eqref{ws}   lies in  the following set 
$$
 \mathcal S_0 = \left\{F_S \left| F^{-1}_{-S}{(u)} =(k_{s,t} h_s'( u)+\ell_{s,t})\id_{\{u\le t\}},~~(s,t)\in D \right.\right\}  
$$
with $$D=\left\{(s,t)\left| 0\le s\le t,~ t\ge \frac{\mu^2}{\mu^2+\sigma^2},~k_{s,t} h_s'(t)+\ell_{s,t}\le 0 \right.\right\}.$$

\begin{proposition} \label{pr-216-1}
For  $h\in \mathcal H_2$ and $x\ge 0$, the problem \eqref{ws} is equivalent to
  \begin{align}
 \sup_{(s,\,t)\in D}~~&~\left\{ g_x(s,t):=-h(s) x - \frac{\mu} t (h(t)-h(s)) + \frac{\sigma_t}{t} \sqrt{t\int_s^t \left( h'(u)\right)^2\, \d u - (h(t)-h(s))^2} \right\}. 
 \label{eq-216-2} 
\end{align}
\end{proposition}
\begin{proof}[Proof.]
 By Lemmas \ref{lm-worstdf} and \ref{lm-31}, we have the problem \eqref{ws} is equivalent to
 \begin{align} 
    \sup_{(s,\,t)\in D}  ~&~ h(s) \max\{-x,\ell_{s,t}\} + \int_s^t \max\left\{ k_{s,t} h_s'(u) + \ell_{s,t} ,-x\right\} h'(u) \d u \label{eq-226-7} 
\end{align}
in the sense that  the worst case distribution of $-S$ of the problem \eqref{ws} is  $F_{-S}^{-1}(u)=(k_{s^*,t^*} h_{s^*}'( u)+\ell_{s^*,t^*})\id_{\{u\le t^*\}}$ if $(s^*,t^*)$ is the optimal solution of problem \eqref{eq-226-7}. Note that by  Lemma \ref{lm-worstdf}, it suffices to consider  $(s,t)\in D$ such that  $\ell_{s,t}\le -x$ if $s>0$ and $-x\le k_{s,t}h'_s(s+)+\ell_{s,t}$ if $s<t$. It then follows immediately that problem \eqref{eq-226-7} is  equivalent to 
\begin{align} 
    \sup_{(s,\,t)\in D}  ~&~\left\{-h(s) x + \int_s^t  ( k_{s,t} h_s'(u) + \ell_{s,t} ) h'(u) \d u = g_x(s,t)\right\}.  
\end{align} 
 Therefore,   the problem  \eqref{ws}  is equivalent to the problem \eqref{eq-216-2}.
   Thus, we complete the proof. 
\end{proof}

 To solve the problem \eqref{eq-216-2}, we give the following observations. 
It is obvious that the objective function $g_x(s,t)$ is continuous in $(s,t)\in D$.
\begin{lemma}\label{lm-3-7}
For $x\ge 0$,   denote  by $(s_0(x),t_0(x))$ one  maximizer of $g_x(s,t)$ over $D$. Then
\begin{itemize} \item  [(a)] $s_0(x)$ is decreasing in $x\ge 0$.
\item [(b)] $g_x(s_0(x),t_0(x)) =\sup_{(s,t)\in D} g_x(s,t)$ is continuous in $x\ge 0$.
\item [(c)] $\lim_{x\to0} s_0(x) =\frac{\mu^2}{\mu^2+\sigma^2}$ and thus the set $D$ can be replaced by $$
D' = \left\{(s,t)\left| 0\le s\le \frac{\mu^2}{\mu^2+\sigma^2} \le t\le 1,~k_{s,t} h_s'(t)+\ell_{s,t}\le 0 \right.\right\}.
$$
\item [(d)] If $s_0(x)=s^*$, then $t_0(x)=t^*$, where $t^*$ is given by \eqref{eq:tsatr}.
\end{itemize}
\end{lemma}
\begin{proof}[Proof.]
(a) Denote by  
  $ g_x(s,t) =  -h(s)x + g_0(s,t)$.
 Note that $g_x(s,t)$ is continuous in $(s,t)\in D$ and $D$ is a compact set. Thus, the maximizer is attainable.  For $x<y$, consider the following  problem
\begin{align}\label{eq-20219-3}
    \max_{(r,t_1)\in D} g_{x}(r,t_1)+ \max_{(s,t_2)\in D} g_{y}(s,t_2) = \max_{(r,t_1,s,t_2)\in D\times D} \{g_{x}(r,t_1)+ g_{y}(s,t_2)\}.
\end{align} 
Denote by $(r^*,t_1^*,s^*,t_2^*)$ one maximizer of the   problem \eqref{eq-20219-3}. 
Note that $g_{x}(r^*,t_1^*)+ g_{y}(s^*,t_2^*) \ge g_{x}(s^*,t_2^*)+ g_{y}(r^*,t_1^*)$, that is,
$ 
 xh(r^*) +   yh(s^*)   \ge  xh(s^*) +  yh(r^*)  .
$ 
This implies 
$
 (x -y) (h(r^*)-h(s^*)) >0,
$
that is, $h(r^*)<h(s^*)$. It then follows  that $r^*<s^*$.  Hence,  $s_0(x)$ is decreasing in $x$.

   (b) The continuity of  $x\mapsto g_{x}(s_0(x),t_0(x))$ follows immediately by noting that $x<y$, 
\begin{align*}
    |g_{y}(s_0(y),t_0(y)) -g_{x}(s_0(x),t_0(x)) | 
   & \le \max_{(s,t)\in D} | g_x(s,t) -   g_y(s,t)| \le y-x.
\end{align*} 

(c)   For $x\ge 0$, denote by $F_x$
 the worst-case distribution of the problem \eqref{ws} and let $S_x\sim F_x$. By Lemma \ref{lm-worstdf}, $S_x$ takes constant value when it is larger than $x$. As $x\to0$,   the distribution $F_x$  converges to   $F_0 := (1-s)\delta_0 +s\delta_{\mu/s}$.  
Note that $S_x$ converges to $S_0$ in $L^\infty$-norm and $s_0(0)=\frac{\mu^2}{\mu^2+\sigma^2}$. We have $\lim_{x\to0} s_0(x) =s_0(0)$ and the set $D'$ follows immediately by the monotonicity of $s_2.$ 

(d) First note that for given $s$, $g_x(s,t)$,  or equivalently, $g_0(s,t)$ is increasing in $t\ge {\mu^2}/({\mu^2+\sigma^2})$.  To see it, note that for $(s,t)\in D$,
\begin{align}
\frac{g_0(s,t)}{1-h(s)}~=~~\sup_{F}&~ \int_0^t F^{-1}(u) \d h_s(u) \label{eq-0305-4}\\
 {\rm subject~to}&~ \int_0^t F^{-1}(u)\d u =-\mu ,~~ \int_0^t( F^{-1}(u))^2\d u =\mu^2 +\sigma^2. \nonumber
\end{align}
By $(s,t)\in D$,  the optimal $F$ satisfies $ F^{-1} \le 0$ on $[0,t]$ and thus can be extended to $[0,t']$ as a feasible solution to problem \eqref{eq-0305-4} with $t$ replaced by $t'>t$. It follows that the optimal value of the  problem \eqref{eq-0305-4}, $\frac{g_0(s,t)}{1-h(s)}$,  is increasing in $t$,
and thus, $g_0(s,t)$  is  increasing in $t$.
Moreover, by standard manipulation,    $k_{s,t} h'_s(t)+\ell_{s,t}\le 0$
is equivalent to
$$
w(t):=   (1+r^2)(th_{s}'(t)-h_{s}(t))^2 - \int_0^1 (h_{s}'(u)-h_{s}(t))^2\d u\le0. 
$$
By taking derivative of $w$, we have 
$w$ is  increasing. Therefore,
the optimal solution of $t$ of $\max_{t}g_x(s,t)$ with fixed $s$ is the largest $t$ such that $w(t)\le 0$, that is,    $t^*$  is defined by  \eqref{eq:tsatr}. 
This completes the proof.   
\end{proof}

Now we are ready to prove Theorem \ref{main-th}. 
\begin{proof}[Proof of Theorem \ref{main-th}.]
For $x\ge 0$, denote by $(s_0(x),t_0(x))$ one  maximizer of $g_x(s,t)$ over $D$ and $A=\{x\ge 0| ph(s_0(x))=c\}.$
Then by Lemma \ref{lm-3-7}, the optimal value of problem    \eqref{ws} is
\begin{align*}
T(x)&:=pg_x(s_0(x),t_0(x)) +cx= (c-ph(s_0(x))) x  + pg_0(s_0(x),t_0(x)). 
\end{align*} 
\begin{itemize}
\item [(a)] If $A\not=\emptyset$, then we assert that the minimizer of $T(x)$ satisfies  
$ 
 ph(s_0(x)) =c.
$ 
To see it, first note that by  that $h$ is increasing and $s_0$ is  decreasing, the set $A$ is an interval. Denote by
$$
 x_1= \inf\{x\ge 0| ph(s_0(x))=c\}~~{\rm and}~~x_2=\sup\{x\ge 0| ph(s_0(x))=c\}.
$$
We first show $T(x)$ is decreasing in $0\le x\le x_1$ and increasing in $x\ge x_2$ by contradiction. 
 Suppose first $T(x)$ is not decreasing in $x<x_1$. Then there exist $x <y \le x_1$ such that $T(x)<T(y)$, that is,
\begin{align*}
(c-ph(s_0(x)) )x  + pg_0(s_0(x),t_0(x))  & < (c-ph(s_0(y))) y  + pg_0(s_0(y),t_0(y))\\
& \le  (c-ph(s_0(y))) x  + pg_0(s_0(y),t_0(y)),
\end{align*}
where the second inequality is due to $ c-ph(s_0(y)) \le  0$.  This implies
$ 
g_x(s_0(x),t_0(x)) 
< g_{x}(s_0(y),t_0(y))
$ 
which yields a contradiction with the maximality of $(s_0(x),t_0(x))$. Hence, we have $T(x)$ is decreasing in $x\le x_1$.  The monotonicity of $T$ in $x>x_2$ follows by similar arguments.
Therefore, we have the minimizer of $T(x)$ is the points satisfying $ph(s_0(x))=c$ and in this case, $$T(x)=pg_0(s_0(x),t_0(x))=- \frac{\mu p} t (h(t)-h(s)) + p\frac{\sigma_t}{t} \sqrt{t\int_s^t \left( h'(u)\right)^2\, \d u - (h(t)-h(s))^2}.$$

\item [(b)] If $A=\emptyset$, we have $ph(s_0(x)) < c$  for all $x\ge 0$ or $ph(s_0(x)) > c$  for all $x\ge 0$.  By the monotonicity of $s_0(x)$, we have $\lim_{x\to\infty} s_0(x)=0$. Otherwise  there exists $\delta>0$ such that $s_0(x)>\delta$ for any $x$, and then $\delta (-x)^2 > \mu^2+\sigma^2$ for some $x$.  This implies there exists $x\ge 0$ such that  $ph(s_0(x))<c$.
 Then we have $ph(s_0(x)) < c$  for all $x\ge 0$.  By similar arguments in Case (a), we have $T(x)$ is increasing in $x\ge0$ and thus the optimal value of $x$ is $x^*=0.$
\end{itemize}
Combining the above analysis, we have the following result.
\begin{itemize}
 \item [(i)]  If $r \ge \sqrt{{1}/{s^*} - 1}$, or equivalently, {$h\left(  \frac1{1+r^2}\right)\le \beta$,} then  by Lemma \ref{lm-3-7} (c), we have $\lim_{x\to 0} (c-ph(s_0(x)) = c-ph\left(  \frac1{1+r^2}\right)\ge 0$.
This implies  $T(x)$ is increasing in $x\ge 0$. Therefore, the optimal value of $x$   is $x^*=0$. 
The optimal value and worst case distribution follows from that 
     when $x=0$, the inner problem of  problem  \eqref{wc-0} is equivalent to
 $ 
  p \sup_{F\in \mathcal {S}(\mu,\sigma)} \rho_h^F(-S\vee 0) =p\rho_h^F(0) =0.
 $

\item[(ii)] 
If  $r < \sqrt{{1}/{s^*} - 1}$,  or equivalently,  {$h\left(  \frac1{1+r^2}\right) > \beta$}, 
then the optimal value of $x$ is those points satisfying $ph(s_0(x))=c$  and for such $x$, the worst case distribution $G$ of the problem \eqref{ws} satisfies $ -G^{-1}(1-\alpha) =G_{s^*}^{-1}(\alpha)$ defined by \eqref{eq-20122-2},
and thus, the problem \eqref{ws} is
 \begin{align} \label{20200229-10}
 \sup_{F\in\mathcal S(\mu,\sigma)} \rho_{h_{s^*}}^F(-S\vee -x)
 \end{align}
 where  $s^*= h^{-1}(\beta)$.
We then consider the following two subcases.
\begin{itemize}

\item [(i.a)] 
If
  $
 -\mu + \sigma \frac{h'(1)-(1-\beta)}{\Delta_{s^*,1}}\le 0,
$  or equivalently, $r \leq \frac{\Delta_{s^*,1}}  { h'(1)-(1-\beta)}$,  then
 the optimal value of \eqref{20200229-10} is attained at the worst case distribution  $G$ given by \eqref{eq:wcdistributoin-2}
as $G$  is  a feasible solution to  problem \eqref{ws} and a worst-case distribution of  the following problem
$
 \sup~\rho_{h_{s^*}}^F(-S\vee -x)~
  {\rm subject~to} ~ \E^F[S] =\mu,~
  {\rm Var}^F(S)=\sigma^2.
$ It then follows the optimal value.
 Note that for $x$ with $ph(s_0(x))=c$, the maximizer $s_0(x)\in (0,1/(1+r^2))$ and $t_0(x)=1$, that is, 
  \begin{align*}
g_x(s^*,t)   &= \max_{s\in (0,1/(1+r^2))}g_x(s,1) \\
& =\max_{s\in (0,1/(1+r^2))}-\mu-(x-\mu)h (s) +\sigma  {\sqrt{\int_{s}^1(h'(u))^2\d u -(1-h({s}))^2} }.
\end{align*}
Note that $s\mapsto g_x(s,1)$ has left- and right-derivatives everywhere, and thus, $s_0(x)$ satisfies
\begin{equation}\label{eq-200301-1-1}
\frac{\partial^- }{\partial s}g_x(s,1)|_{s=s^*} \le 0\le  \frac{\partial^+ }{\partial s}g_x(s,1)|_{s=s^*}
~~{\rm or}~~
\frac{\partial^- }{\partial s}g_x(s,1)|_{s=s^*} \ge 0\ge  \frac{\partial^+ }{\partial s}g_x(s,1)|_{s=s^*},
\end{equation}
where $({\partial^- }/{\partial s})g $ or $({\partial^+ }/{\partial s})g$ denote  the left and right-derivative of $g$, respectively. 
By calculating the left-and right-derivatives,  we have the optimal solution is given by \eqref{eq:solutionx-2}.

\item  [(ii.b)] 
If $r > \frac{\Delta_{s^*,1}}  { h'(1)-(1-\beta)}$, 
note that for those $x$ satisfying $ph(s_0(x))=c$, by (d) of Lemma \ref{lm-3-7}, the optimal value of $t$ is  $t^*$ and 
$$
 g_x(s,t^*) = -h(s)x -\frac{\mu}{t^*} (h(t^*)-h(s)) + \frac{\sigma_{t^*}}{t^*}\sqrt{t^*\int_s^{t^*} \left( h'(u)\right)^2\, \d u - (h(t^*)-h(s))^2}.
$$
 The optimal value of $x$ and the worst case distribution follow by similar arguments in Case (ii). 
 \end{itemize}
 \end{itemize}
Combining the above two cases, we complete the proof. 
\end{proof}

\section{Multi-Product Risk-Averse Newsvendor}
Extending the risk-averse newsvendor model to multiple items is natural but analytically challenging. In the risk-neutral case, the extension is straightforward: the expected total profit decomposes additively across items. In contrast, incorporating risk measures introduces significant complexity—both analytically and computationally (see \cite{CRZ11}). It is worth noting that  while \cite{CRZ11} derive closed-form approximations, their results rely on the assumption of independent demand across items. We consider the following multi-product risk-averse newsvendor model:
\begin{equation}
\label{eq-0413-1}
\rho_h^{\rm wc}(\boldsymbol{\mu},\boldsymbol{\sigma})=\inf_{\boldsymbol{x}
\in \R_{+}^n
} \sup_{F\in \cal S(\boldsymbol{\mu},\boldsymbol{\sigma})} \rho_h^F\left(\sum_{i=1}^n\psi _i (x_i,p_i,c_i)\right), 
\end{equation}
where $\rho_h^F$ denotes the distortion functional $\rho_h$ evaluated under the distribution $F$ of $\boldsymbol{S} = (S_1, \ldots, S_n)$, and:
$$
\psi_i(x,p_i,c_i)= p_i\max\{-S_i,-x\} + c_ix,~~i=1,\ldots,n,
$$
$${\cal S}(\mu,\sigma) =\{F~\text{cdf on}~[0,\infty)^n: \E^F[\boldsymbol{S}]=\boldsymbol{\mu},~  {\rm Var}^F(S_i)=\sigma^2_i,~~i=1,\ldots,n\}.$$
The model captures distributional ambiguity and allows for arbitrary dependence structures across item demands.

Notably, in the above multi-product setting, the optimal order quantities can be determined by solving single-product problems separately—a result that may initially appear counterintuitive.
\begin{theorem}
Given $\boldsymbol{\mu} \in \R^n$, $\boldsymbol{\sigma}\in \R_{+}^n$ and a convex distortion function $h$, the optimal value of  problem \eqref{eq-0413-1} is
\begin{equation}\label{20-414-2}
\rho_h^{\rm wc}(\boldsymbol{\mu},\boldsymbol{\sigma}) 
=\sup_{F\in \cal S(\boldsymbol{\mu},\boldsymbol{\sigma})} \rho_h^F\left(\sum_{i=1}^n\psi _i (x_i^*,p_i,c_i)\right) 
=\sum_{i=1}^n \rho_h^{\rm wc}(\mu_i,\sigma_i), 
\end{equation}
where $x_i^*$ and $\rho_h^{\rm wc}(\mu_i,\sigma_i)$ are given by Theorem \ref{main-th} with $\mu=\mu_i$ and $\sigma=\sigma_i$, $i\in[n]$. 
\end{theorem}
\begin{proof}[Proof.]
By the subadditivity of $ \rho_h$, it holds
$ \rho_h^F\left(\sum_{i=1}^n\psi _i (x_i,p_i,c_i)\right)
\le \sum_{i=1}^n\rho_h^F\left(\psi _i (x_i,p_i,c_i)\right),$ 
 and thus, 
\begin{equation}\label{20-414-1}
\rho_h^{\rm wc}(\boldsymbol{\mu},\boldsymbol{\sigma}) \le \sum_{i=1}^n \rho_h^{\rm wc}(\mu_i,\sigma_i).
\end{equation}
For $i\in[n]$,  consider the problem $\sup_{F\in {\cal S}(\mu_i,\sigma_i)} \rho_h^F\left(\psi _i (x_i,p_i,c_i)\right) $ and
let $F^*_i$ be one of its worst-case distributions. Define $F^*(\boldsymbol{s}) =\min\{F_1^*(s_1),\ldots, F_n^*(s_n)\}$,~$\boldsymbol{s}\in \R_+^n.$ 
Let $\boldsymbol{S}^*=(S_1^*,\ldots,S_n^*)$ be a random variable such that its distribution is $F^*$. That is, $(S_1^*,\ldots,S_n^*)$ are comonotonic and $S_i^*\sim F_i^*$, $i\in[n]$. Then one can verify that $ \psi _i^* (x_i,p_i,c_i)=p_i\max\{-S_i^*,-x_i\} + c_ix_i$, $i\in[n]$ are  comonotonic. Then by the comonotonic addivity of the $\rho_h$, we have
$$ 
 \sup_{F\in \cal S(\boldsymbol{\mu},\boldsymbol{\sigma})}\rho_h^F\left(\sum_{i=1}^n\psi _i (x_i,p_i,c_i)\right)
 \ge  \rho_h^{F^*}\left(\sum_{i=1}^n\psi _i (x_i,p_i,c_i)\right) = \sum_{i=1}^n\rho_h^{F_i^*}\left(\psi _i^* (x_i,p_i,c_i)\right)=\sum_{i=1}^n \rho_h^{\rm wc}(\mu_i,\sigma_i).
$$  
This combined with \eqref{20-414-1} implies \eqref{20-414-2}. Then the value of $ (x_1^*,\ldots,x_n^*)$ follows from Theorem \ref{main-th}. 
\end{proof}

This separability stems from a distinctive property of distortion functionals: their additivity under comonotonicity allows exact risk aggregation without assuming independence or inducing artificial diversification. As a result, even under arbitrary demand dependencies, multi-product problems reduce to single-item problems—underscoring both the analytical strength and practical scalability of our framework.


\end{document}